\documentclass[a4paper,10pt]{article}

\usepackage[utf8]{inputenc}

\usepackage{amssymb}
\usepackage{amsthm}
\usepackage{amsmath}
\usepackage{anysize}
\usepackage{url}
\usepackage{color}
\usepackage{subfigure,graphicx}
\usepackage{overpic}
\usepackage{epstopdf}

\newcommand{\bitem}{\begin{itemize}}
\newcommand{\eitem}{\end{itemize}}
\newcommand{\beq}{\begin{equation}}
\newcommand{\eeq}{\end{equation}}
\newcommand{\beqn}{\begin{eqnarray*}}
\newcommand{\eeqn}{\end{eqnarray*}}

\newcommand{\cC}{{\cal C}}

\newcommand{\cE}{{\cal E}}

\newcommand{\argmin}{\mbox{argmin}}

\newcommand{\ip}[2]{\left\langle#1,#2\right\rangle}

\newcommand{\norm}[1]{\|#1\|}
\def\supp{{\text{\rm supp }}}
\def\diag{{\text{\rm diag}}}

\def\N{\mathbb{N}}

\def\Z{\mathbb{Z}}

\def\R{\mathbb{R}}
\def\RR{\mathbb{R}}

\def\C{\mathbb{C}}

\def\epsilon{\varepsilon}

\newtheorem{theorem}{Theorem}[section]
\newtheorem{lemma}[theorem]{Lemma}
\newtheorem{cor}[theorem]{Corollary}
\newtheorem{definition}[theorem]{Definition}
\newtheorem{remark}[theorem]{Remark}

\newtheorem{prop}[theorem]{Proposition}

\definecolor{gray}{gray}{.5}

\newcommand{\ms}[1]{#1}
\newcommand{\pg}[1]{#1}
\newcommand{\msch}[1]{#1}

\author{Philipp Grohs, Sandra Keiper, Gitta Kutyniok, and Martin Sch\"{a}fer}

\title{$\alpha$-Molecules}

\begin{document}

\maketitle

\begin{abstract}
Within the area of applied harmonic analysis, various multiscale systems such as wavelets, ridgelets, curvelets,
and shearlets have been introduced and successfully applied. The key property of each of those systems are their
(optimal) approximation properties in terms of the decay of the $L^2$-error of the best $N$-term approximation
for a certain class of functions. In this paper, we introduce the general framework of $\alpha$-molecules, which
encompasses most multiscale systems from applied harmonic analysis, in particular, wavelets, ridgelets, curvelets,
and shearlets as well as extensions of such with $\alpha$ being a parameter measuring the degree of anisotropy, as a means to allow a
unified treatment of approximation results within this area. Based on an $\alpha$-scaled index distance, we first
prove that two systems of $\alpha$-molecules are almost orthogonal. This leads to a general methodology to transfer
approximation results within this framework, provided that certain consistency and time-frequency localization
conditions of the involved systems of $\alpha$-molecules are satisfied. We finally utilize these results to
enable the derivation of optimal sparse approximation results \msch{for} a specific class of cartoon-like functions
by sufficient conditions on the `control' parameters of a system of $\alpha$-molecules.
\end{abstract}

\begin{center}
\end{center}

\section{Introduction}

Applied Harmonic Analysis is by now one of the most thriving areas within applied mathematics. This success is
mainly due to the range of efficient multiscale systems it provides, which are today employed for a variety
of real-world applications. Just think of the first and hence `oldest' in this list, which are {\em wavelet systems}
\cite{Dau92}. In the world of imaging science, these systems are today utilized, for instance, for image restoration tasks
\cite{CDOS12} and in the world of partial differential equations, they were key to developing provably optimal
solvers for elliptic equations \cite{CDD01}, to name a few. Their crucial property is to optimally sparsely
approximate functions governed by point singularities -- in the sense of the decay rate of the $L^2$-error of the best $N$-term
approximation.

Following this grand opening, next came {\em ridgelets}, introduced by Cand\`{e}s in his PhD thesis \cite{Can98} and further developed
jointly with Donoho \cite{CD99}, which are perfectly suited for encoding ridge-like singularities appearing, for
instance, in tomography. Since it is today a general belief that images -- as well as, for instance,
solutions to transport equations -- are governed by curvilinear singularities such as edges, Cand\`{e}s and
Donoho then introduced {\em curvelets} \cite{CD04}, which were the first system to provide provably optimally sparse
approximations of a suitable model situation, thus justifiably called the second breakthrough after wavelets.
Some years later, {\em shearlets} were introduced mainly by Guo, Labate,  and one of the authors \cite{GKL05} as a
system capable of providing the same approximation behavior as curvelets, but having the advantage of providing a
unified treatment of the continuum and digital realm; nowdays used, for instance, for imaging applications
(see, e.g., \cite{EL2012}) or for solvers of transport equations \cite{DHKSW14}. And these are just a selection of multiscale
systems being developed in the area of applied harmonic analysis.

As one can see, each of those multiscale systems in $L^2(\mathbb{R}^2)$ satisfies distinct optimal sparse approximation properties for
a particular class of functions. Some of those such as curvelets and shearlets even for the same class. Besides
the aforementioned applications, such sparse approximation properties are also key to the novel area of compressed
sensing \cite{CRT06b,Don06c}, which requires a sparsifying system for the considered data. And indeed, systems
from applied harmonic analysis have already been extensively utilized for this task, see, for instance,
\cite{Donoho2010c,GK14}.

Analyzing the different constructions of such systems, one cannot fail to observe certain similarities, which
appear due to the fact that a guiding principle in applied harmonic analysis is to develop multiscale systems
based on their partition of Fourier domain as well as by utilizing certain operators (scaling-, translation-, etc.)
to generating functions. A careful observer also does not fail to notice that certain proofs such as for sparse
approximation properties of band-limited curvelets and \msch{shearlets} are quite resembling from a structural viewpoint.

Thus, one has to ask whether a general framework is acting in the background of all those results. Bringing this
to light would for the first time enable a common treatment of multiscale systems in applied harmonic analysis,
in particular, with respect to their approximation behavior, thereby enabling, for instance, transfer of known
results from one system to another or categorization of multiscale systems with respect to their sparsity behavior.
In this paper, we will introduce such a general framework which we coin {\em $\alpha$-molecules} for reasons to
be explained in the sequel.

\subsection{Towards a General Framework}

Introducing a general framework, the first step shall always be to pause and contemplate which list of desiderata we
expect this framework to satisfy. In our case \pg{the} introduced framework
shall foremost satisfy the following properties:
\bitem
\item[(D1)] Encompass most known multiscale systems within the area of applied harmonic analysis.\\[-4ex]
\item[(D2)] Allow the construction of novel multiscale systems.\\[-4ex]
\item[(D3)] Allow a categorization of systems with respect to their approximation behavior.\\[-4ex]
\item[(D4)] Enable a transfer of (sparse approximation) results between the systems within this framework.\\[-4ex]
\item[(D5)] Enable the derivation of approximation results by \pg{easy-to-verify} conditions on certain parameters associated with a system.
\eitem

Let us start by considering the two representation systems of curvelets and shearlets, which exhibit similar approximation properties
in the sense that they both exhibit an optimal decay rate of the $L^2$-error of best $N$-term approximation for the class of
so-called {\em cartoon-like functions}, \pg{which are roughly speaking} compactly supported functions that are $C^2$ apart from a
$C^2$-discontinuity curve. \pg{The common bracket in the construction of curvelets and shearlets} is {\em parabolic scaling}, i.e., a scaling matrix of the type
$\diag(s,s^{1/2})$, $s>0$ which leaves the parabola invariant. In fact, this type of scaling, which produces functions
with essential support {\em width $\approx$ length$^2$}, is specifically adapted to the fact that the model is
based on a $C^2$-discontinuity curve; a heuristic argument can be easily derived by expanding the curve parametrized
by $(E(x_2),x_2)$ with $E(0)=0 = E'(0)$ in a Taylor series in $x_2=0$ and using that $E($length of generator$) = $ width
of generator, when centering the generating function in the origin. Those considerations eventually led to the
framework of parabolic molecules \cite{Grohs2011}.

In this paper, we however aim much higher, envisioning to develop a framework which, for instance, also includes
wavelets and ridgelets. Key to our work and also the reason of the term `$\alpha$-molecule' is the observation
that a distinct property of all multiscale systems is the degree of anisotropy of their scaling operators.
\pg{Whereas} wavelet systems rely on isotropic scaling, i.e., the scaling matrix $\diag(s,s)$, ridgelets are based on the most
aniostropic scaling imaginable, which is $\diag(s,1)$. Thus, a system within the proposed framework has to
be associated with a particular {\em ($\alpha$-)scaling} of the type
\[
\begin{pmatrix} s & 0 \\ 0 & s^\alpha \end{pmatrix}, \quad s>0,
\]
the parameter $\alpha$ ranging from $\alpha=1$ (wavelets) over $\alpha = \frac12$ (curvelets and shearlets) to $\alpha = 0$
(ridgelets). The fact that, in addition, the expression `molecule' is to some extent standard in the literature of harmonic
analysis (see, for instance, \cite{FJW91}), explains the terminology {\em framework of $\alpha$-molecules}.

\subsection{The Framework of $\alpha$-Molecules}

Aiming to satisfy (D1)--(D5), the introduced systems of $\alpha$-molecules in Definition \ref{defi:alphamolecules}
comprise the following ingredients:
\begin{itemize}
\item Each system can have a different indexing set, which is then -- for the sake of a unified definition and to
enable a comparison of systems of $\alpha$-molecules -- mapped to a common parameter space.\\[-4ex]
\item $\alpha$-Scaling, translation, and rotation operators are applied to a set of generating functions which
provides maximal flexibility by allowing a different generator for each index.\\[-4ex]
\item Certain control parameters determine the time-frequency localization as well as the (almost) vanishing
moment conditions of the generating functions.
\end{itemize}
Those ingredients ensure (D1) to be fulfilled as well as (D2).

A key property of systems of $\alpha$-molecules is the almost orthogonality of each pair, made precise in
Theorem \ref{thm:almostorth}; or in other terms, the almost diagonality of the associated cross-Gramian matrix.
\ms{Using an extension of the \msch{concept} of sparsity equivalence introduced in \cite{Grohs2011}, which provides
a notion for two systems of $\alpha$-molecules to \msch{possess} a similar sparsity behavior, the almost orthogonality yields
sufficient conditions for two systems to be sparsity equivalent in Theorem \ref{theo:sparseeqivmol};
thereby deriving (D3). We note that this is no true equivalence relation, but serves our purposes for the analysis.}

\pg{Desideratum} (D4), i.e., the transfer of sparse approximation properties from one system of $\alpha$-molecules
to another, is closely related to this \ms{notion of sparsity equivalence} \pg{whose} effectiveness will exemplarily be
demonstrated by deriving \ms{a novel sparse approximation result for band-limited $\alpha$-shearlet systems, formulated in Theorem~\ref{theo:bandlimshearapprox}.
In fact, it is a corollary of the more general Theorem~\ref{theo:mainsparsity} in connection with Theorem~\ref{thm:consist}}. 

The derivation of approximation results by conditions on certain parameters associated with a system of
$\alpha$-molecules, i.e., \pg{Desideratum (D5)}, can in general be derived by transferring known approximation
results from one `anchor' system to all other $\alpha$-molecules.
One possibility, which we will present in detail, is the transfer of optimal sparse approximation results of
so-called $\alpha$-curvelets \cite{GKKScurve2014} for a certain extended class of cartoon-like functions,
yielding sufficient conditions on the `control \msch{parameters}' of a system of $\alpha$-molecules to exhibit
the same optimal approximation behavior (cf. Theorem \ref{theo:mainsparsity2}).

\subsection{Expected Impact}

We anticipate our results to have the following impacts:

\begin{itemize}
\item
	\emph{Approximation Theory:} The framework of $\alpha$-molecules now provides a common platform for
studies of approximation behavior of multiscale systems within the area of applied harmonic analysis.
It is flexible enough to enable a transfer of approximation results from one system to another
and to categorize systems by their approximation behavior. It allows
a deep insight into the relation between time-frequency localization and approximation properties,
and we expect it to significantly ease the construction of multiscale systems for function classes,
arising from future technologies.
\item
	\emph{Theory of Function Spaces:} Smoothness spaces associated with a multiscale system\pg{, characterized by the decay of expansion coefficients,} are in a natural way related to the study
of approximation properties. And, in fact, a deep understanding of their structure is crucial, in particular,
for applications in numerical analysis of partial differential equations. In \cite{Grohs2011}, a first
approach to a unified theory for systems based on parabolic scaling was undertaken. We strongly expect the
framework of $\alpha$-molecules to eventually lead to a unified structural treatment of smoothness spaces
associated with all encompassed multiscale systems.
\item
	\emph{Compressed Sensing:} Compressed Sensing relies on the existence of optimal\pg{ly} sparsifying systems
for given data. Systems from applied harmonic analysis have the advantage of coming with a fast transform
and known functional analytic properties, in contrast to systems being generated by dictionary learning
algorithms. Thus, one might envision the general framework of $\alpha$-molecules to provide a wide range
of flexible multiscale systems allowing an adaption to the data at hand by learning certain control
parameters, but still preserving their advantageous functional analytic and numerical properties.
\end{itemize}

\subsection{Outline}

The paper is organized as follows. Section \ref{sec:framework} is devoted to the introduction of the framework of $\alpha$-molecules.
More precisely, based on the most prominent multiscale systems whose definitions are briefly recalled in Subsection \ref{subsec:prominent},
the notion of a system of $\alpha$-molecules is introduced in Subsection \ref{subsec:defalpha}. It is then shown in Section
\ref{sec:examples} that various versions of wavelets, curvelets, ridgelets, and shearlets (in this order) are indeed instances of
$\alpha$-molecules. The analysis of the cross-Gramian of two systems of $\alpha$-molecules showing their almost orthogonality based
on an $\alpha$-scaled index distance is presented in Section \ref{sec:gramian}. This fact is utilized in Section~\ref{sec:approx} to introduce
\ms{the notion of sparsity equivalence for systems of $\alpha$-molecules}, analyze the ability of the framework to transfer sparse approximation
results from one system to another, and at last, provide results on the optimal sparse approximation behavior of $\alpha$-molecules
with respect to a certain class of cartoon-like functions depending on their control parameters. Finally, several highly technical
and lengthy proofs are outsourced to Section \ref{sec:proofs}.

\section{A General Framework for Applied Harmonic Analysis}
\label{sec:framework}

Aiming to introduce a general framework, which encompasses most multiscale representation systems developed
within the area of applied harmonic analysis, we start by reviewing some of the most prominent systems,
namely wavelets \cite{Dau92}, ridgelets \cite{CD99}, curvelets \cite{CD04}, and shearlets \cite{GKL05}.
If the framework shall be meaningful, those systems should undoubtedly be included; serving us as intuition
and guideline for the definition of $\alpha$-molecules.

\subsection{Prominent Multiscale Representation Systems}
\label{subsec:prominent}

Historically correct, we will start with recalling the definition of wavelets. Since the notion of $\alpha$-curvelets
from \cite{GKKScurve2014} allows us to unify the notions of ridgelets and curvelets, we will \pg{then} introduce those,
followed by the definitions of (second generation) curvelets, and then ridgelets. We conclude this subsection \pg{by}
stating the definition of shearlets. Throughout, we will use the version $\widehat{\varphi}(\xi)=\mathcal{F}\varphi(\xi)
=\int_{\R} \varphi(x)e^{-2\pi ix\xi} \,dx$ for the Fourier transform of $f\in L^1(\R^d)$, and extend it in the usual
way to tempered distributions.

\subsubsection{Wavelets}

Of the various wavelet constructions for $L^2(\R^2)$, the tensor product construction (cf. \cite{Woj97}) is the most widely utilized one.
Starting with a given multi-resolution analysis of $L^2(\R)$ with scaling function $\phi^0\in L^2(\R)$ and wavelet $\phi^1\in L^2(\R)$,
the functions $\psi^e\in L^2(\R^2)$ are defined for every index $e=(e_1,e_2)\in E$, where $E=\{0,1\}^2$, as the tensor products
\begin{align}\label{eq:wavegen}
\psi^e=\phi^{e_1}\otimes\phi^{e_2}.
\end{align}
These functions serve as the generators for the wavelet system defined below.
The corresponding tiling of the frequency plane is illustrated in Figure~\ref{fig:wave}.

\begin{definition}\label{def:tensorwavelet}
Let $\phi^0$, $\phi^1\in L^2(\R)$ and $\psi^e\in L^2(\R^2)$, $e\in E$, be defined as above. Further, let $\sigma>1$, $\tau>0$ be fixed sampling
parameters. The associated {\em wavelet system} $W\big(\phi^0,\phi^1;\sigma,\tau\big)$ is then defined by
\begin{align*}
W\big(\phi^0,\phi^1;\sigma,\tau\big)&=\Big\{ \psi^{(0,0)}(\cdot - \tau k) ~:~ k\in\Z^2 \Big\} \cup \Big\{ \sigma^j\psi^e(\sigma^j\cdot - \tau k) ~:~
e\in E\backslash\{(0,0)\},\,j\in\N_0,\,k\in\Z^2 \Big\}.
\end{align*}
\end{definition}

\begin{figure}[ht]
 \centering
 \subfigure{
 \includegraphics[width = .2\textwidth]{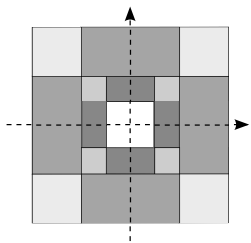}
 }
 \caption{Partition of Fourier domain induced by tensor wavelets.}
 \label{fig:wave}
 \end{figure}

\subsubsection{($\alpha$-)Curvelets}
\label{subsec:curvelets}

In 2002 Cand\`es and Donoho \cite{CD04} introduced the second generation of curvelets, nowadays simply refered to as \emph{curvelets},
\ms{the construction \pg{of which is} based on a parabolic scaling law.}
The idea to allow more general $\alpha$-scaling with $\alpha\in[0,1]$ advocated in \cite{GKKScurve2014}, yields a whole scale of
representation systems, which interpolates between wavelets for $\alpha=1$ and ridgelets for $\alpha=0$. As already discussed in the
introduction, such a general scaling-viewpoint is associated with the scaling matrix
\begin{equation} \label{eq:scalingmatrix}
A_{\alpha,s}=\begin{pmatrix} s & 0 \\ 0 & s^\alpha \end{pmatrix}, \quad s>0.
\end{equation}

We start by defining the radial and angular components separately. For the construction of the radial functions $W^{(j)}$, $j\in\N_0$, let
$\widetilde{W}^{(0)}:\R_+\rightarrow[0,1]$ and $\widetilde{W}:\R_+\rightarrow[0,1]$ be $C^\infty$-functions with the following properties:
\begin{align*}
&\supp\, \widetilde{W}^{(0)} \subset[0,2), \quad\quad \widetilde{W}^{(0)}(r)=1 \quad \mbox{for all $r \in[0,\textstyle{\frac{3}{2}}]$},  \\
&\,\:\supp\, \widetilde{W} \subset(\textstyle{\frac{1}{2}},2), \quad\quad\quad  \widetilde{W}(r)=1 \quad \:\mbox{ for all $r \in[\textstyle{\frac{3}{4}},\textstyle{\frac{3}{2}}]$}.
\end{align*}
Then, for $j\in\N$ and $r\in\R_+$, set
\[
\widetilde{W}^{(j)}(r):=\widetilde{W} (2^{-j} r).
\]
In a final step, for every $j\in\N_0$, we rescale
\[
W^{(j)}(r):=\widetilde{W}^{(j)} (8\pi r) \quad,r\in\R_+,
\]
in order to obtain an integer grid later. Notice, that $2\ge\sum_j W^{(j)}\ge 1$.

Next, we define the angular functions $V^{(j,\ell)}:\mathbb{S}^{1}\rightarrow[0,1]$,
where $\mathbb{S}^{1}\subset\R^2$ denotes the unit circle, $j\in\N$ and the index $\ell$ runs through $0,\ldots,L_j-1$ with
\[
L_j=2^{\lfloor j(1-\alpha) \rfloor}, \quad j\in\N.
\]
We start with a $C^\infty$-function $V:\R\rightarrow[0,1]$, living on $\R$ and satisfying
\[
\supp\, V\subset [-\textstyle{\frac{3}{4}}\pi,\textstyle{\frac{3}{4}}\pi] \quad \mbox{and} \quad
V(t)= 1 \quad  \text{ for all }t\in[-\textstyle{\frac{\pi}{2}},\textstyle{\frac{\pi}{2}}].
\]
For every $j\in\N$, we let $\widetilde{V}^{(j,0)}:\mathbb{S}^{1}\rightarrow[0,1]$ be the restriction of the scaled version
$V(2^{\lfloor j(1-\alpha) \rfloor} \: \cdot)$ of the function $V$ to the interval $[-\pi,\pi]$. Since $[-\pi,\pi]$ can be identified
with $\mathbb{S}^{1}$ via $\varphi:t\mapsto e^{it}$, this yields a function $\widetilde{V}^{(j,0)}$ on $\mathbb{S}^{1}$, which is $C^\infty$.

In order to obtain real-valued curvelets, we symmetrize by
\[
V^{(j,0)}(\xi):=\widetilde{V}^{(j,0)}(\xi) + \widetilde{V}^{(j,0)}(-\xi) \quad\text{ for }\xi\in\mathbb{S}^1.
\]
Then, for each scale $j\in\N$, we define the angles $\omega_j=\pi 2^{-\lfloor j(1-\alpha) \rfloor}$. We next use the notation
\begin{equation} \label{eq:rotation}
R_{\theta}=\begin{pmatrix} \cos(\theta) & - \sin(\theta) \\
 \sin(\theta) & \cos(\theta)
\end{pmatrix},
\quad \theta \in [0,2\pi],
\end{equation}
for the rotation matrix and put $R_{j,\ell}:=R_{\ell\omega_j}$. By rotating $V^{(j,0)}$, for each $\ell=0,1,\ldots,L_j-1$, we finally define
$V^{(j,\ell)}:\mathbb{S}^{1}\rightarrow[0,1]$ by
\[
V^{(j,\ell)}(\xi):=V^{(j,0)}(R_{j,\ell} \xi) \quad\text{ for }\xi\in\mathbb{S}^1.
\]

In order to secure the tightness of the frame we utilize the function
\[
\Phi(\xi):= W^{(0)}(|\xi|)^2 + \sum_{j,\ell} W^{(j)}(|\xi|)^2V^{(j,\ell)}\Big(\frac{\xi}{|\xi|}\Big)^2.
\]
Notice, that $1\le\Phi(\xi)\le8$ for all $\xi\in\R^2$. Next, we combine the radial and angular components together and define the functions
$\psi_0$ and $\psi_{j,\ell}$ on the Fourier side by
\begin{equation} \label{eq:def_curve}
\widehat{\psi}_{0}(\xi):=\frac{W^{(0)}(|\xi|)}{\sqrt{\Phi(\xi)}} \quad \mbox{and} \quad
 \widehat{\psi}_{j,\ell}(\xi)=\frac{W^{(j)}(|\xi|)V^{(j,\ell)}\Big(\frac{\xi}{|\xi|}\Big)}{\sqrt{\Phi(\xi)}}.
\end{equation}
Observe that $\widehat{\psi}_{0}$, $\widehat{\psi}_{j,\ell}\in C^\infty(\R^2)$, and that these functions are real-valued, non-negative,
compactly supported and $L^\infty$-bounded by $1$.

\begin{definition}
\label{defi:curvelets}
Let $\alpha \in [0,1]$, and let $\psi_{0}$ and $\psi_{j,\ell}$ be defined as in \eqref{eq:def_curve}. Then the associated {\em
$\alpha$-curvelet system} $C_\alpha(W^{(0)},W,V)$ is defined by
\begin{align*}
C_\alpha(W^{(0)},W,V)=\Big\{ \psi_{0,k} ~:~ k\in\Z^2 \Big\} \cup \Big\{ \psi_{j,\ell,k} ~:~ j\in\N,\,k\in\Z^2,\,\ell\in\{0,1,\ldots,L_j-1\} \Big\},
\end{align*}
where, for $j\in\N$, $\ell\in\{0,1,\ldots,L_j-1\}$, $k\in\Z^2$,
\[
\psi_{0,k}:=\psi_0(\cdot-k)  \quad \mbox{and} \quad \psi_{j,\ell,k}:=2^{-j(1+\alpha)/2} \cdot \psi_{j,\ell}(\cdot - x_{j,\ell,k} )
\;\: \mbox{with }x_{j,\ell,k}=R^{-1}_{j,\ell}A^{-1}_{\alpha,2^{j}}k.
\]
\end{definition}

It was shown in \cite{GKKScurve2014} that $C_\alpha(W^{(0)},W,V)$ constitutes a tight frame for $L^2(\R^2)$.
The induced frequency tiling for different $\alpha\in[0,1]$ is depicted in Figure~\ref{fig:alphacurve}.

\begin{figure}[ht]
 \centering
 \subfigure{
 \includegraphics[width = .2\textwidth]{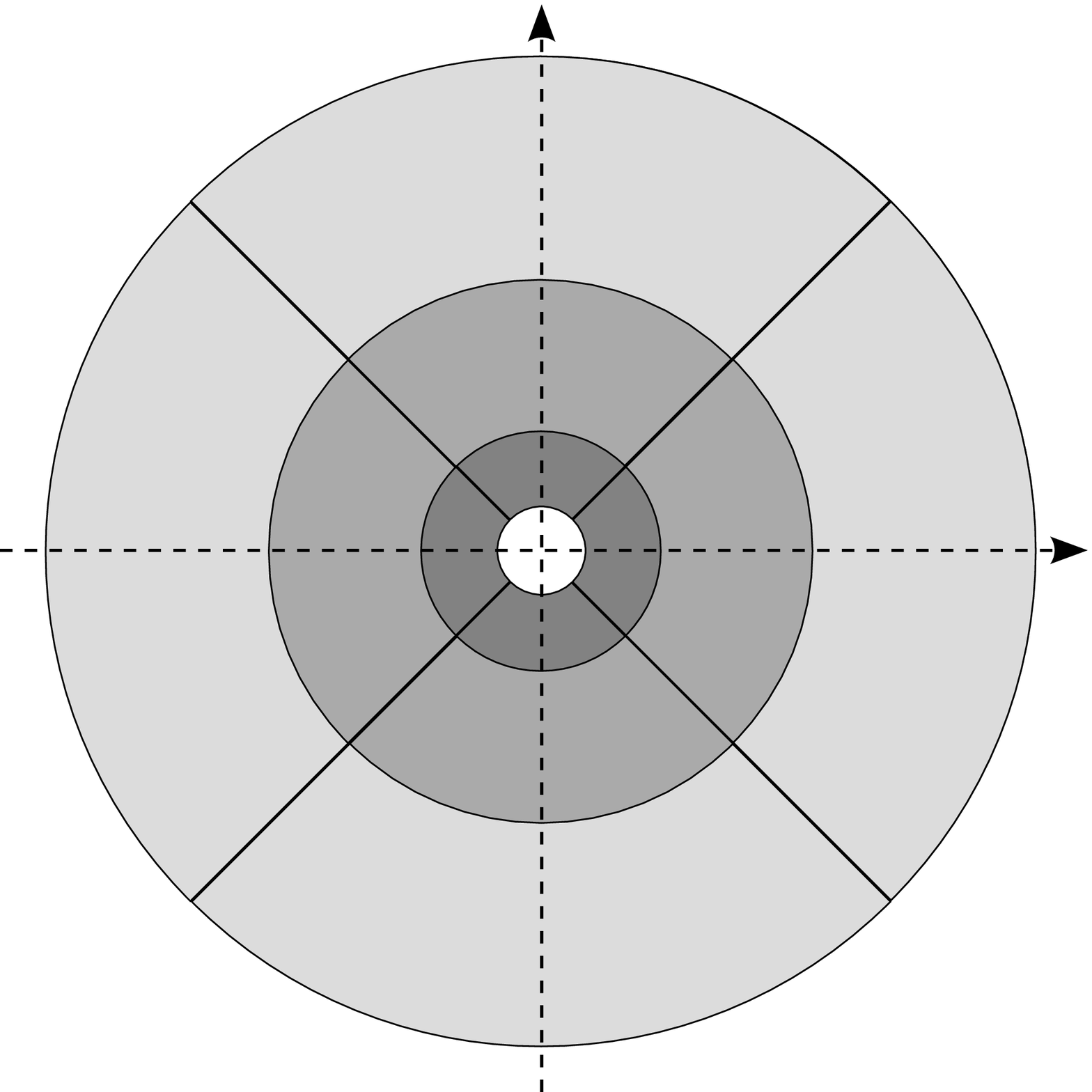}
 \put(-53,-17){(a)}
 }
 \qquad
 \subfigure{
 \includegraphics[width = .2\textwidth]{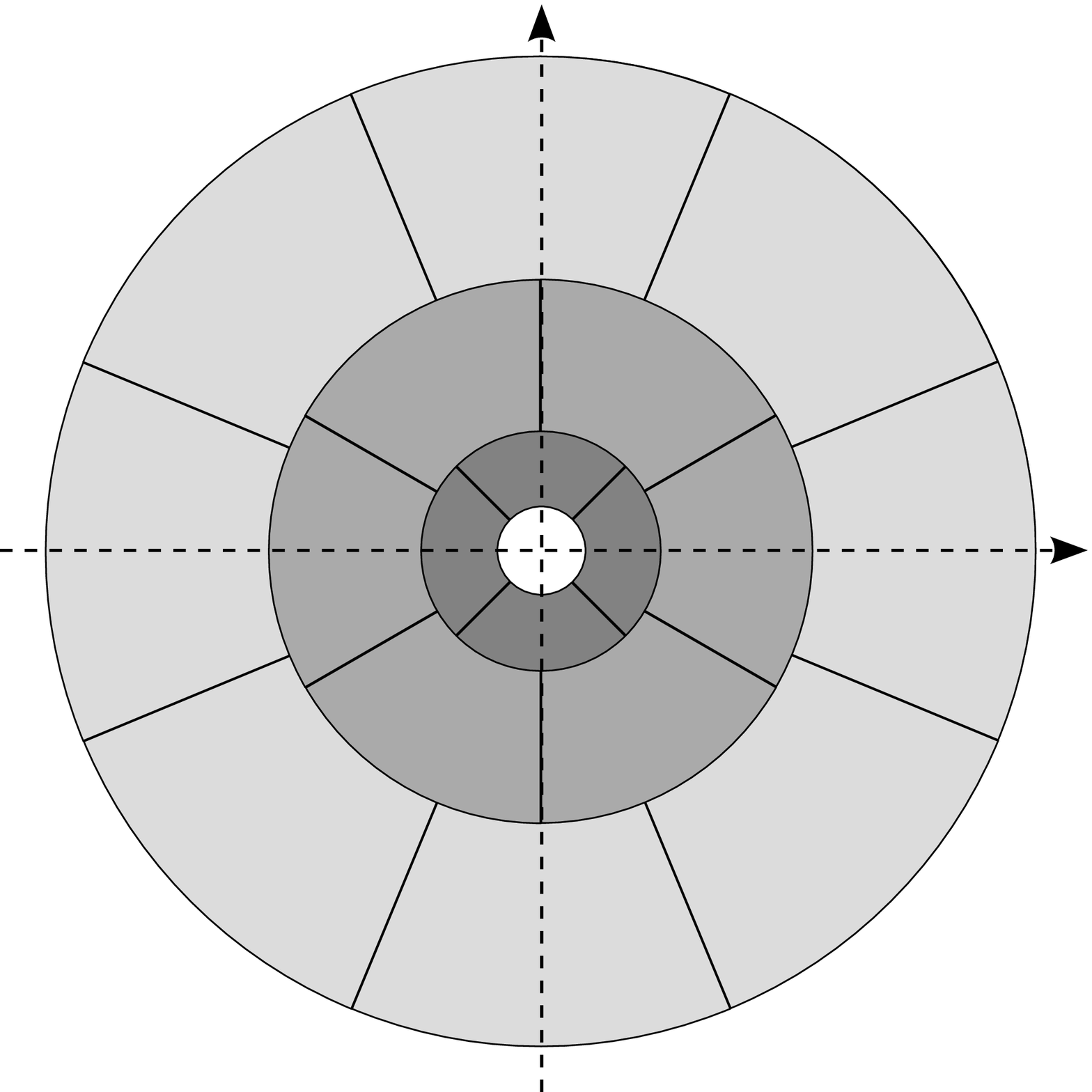}
 \put(-52,-17){(b)}
 }
 \qquad
 \subfigure{
 \includegraphics[width = .2\textwidth]{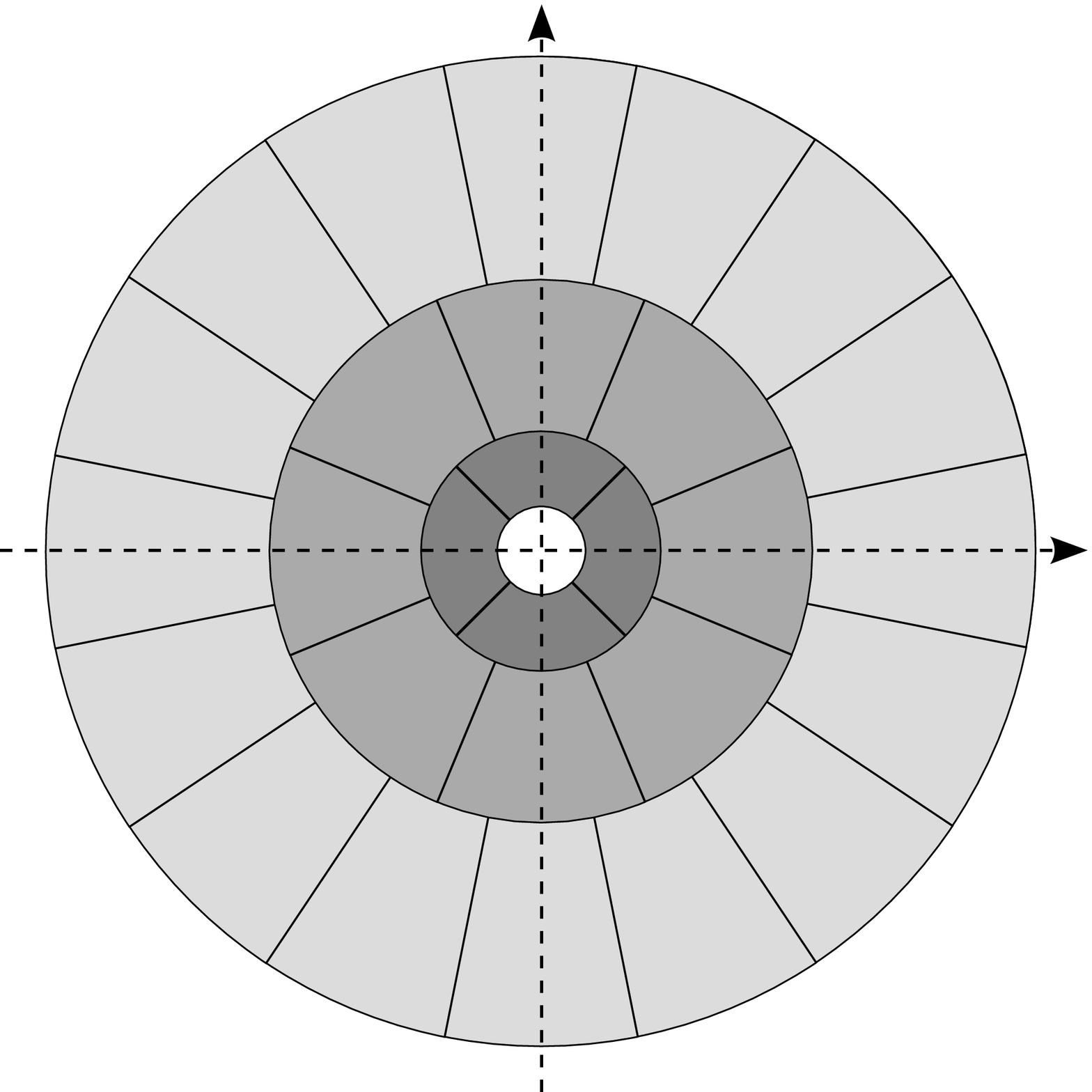}
 \put(-51,-17){(c)}
 }
 \caption{Partition of the Fourier domain induced by $\alpha$-curvelets for (a): $\alpha=1$, (b): $\alpha=1/2$, and  (c): $\alpha=0$.}
 \label{fig:alphacurve}
 \end{figure}

\begin{remark}
\ms{
The definition of $\alpha$-curvelets given in Definition \ref{defi:curvelets} is closely related to and
inspired by the classical (second generation) curvelets from \cite{CD04}.
The original system is obtained by a slight modification of the angular tiling of the construction in the case $\alpha = \frac{1}{2}$.
In contrast to $\frac{1}{2}$-curvelets, the resolution of the angular tiling is doubled at every other scale and remains fixed in between,
as depicted in Figure~\ref{fig:secgen}. In addition, the orientations of the single functions at every scale are chosen in a slightly different manner.
The reader might want to compare this to the frequency tiling of $\frac{1}{2}$-curvelets, Figure~\ref{fig:alphacurve}(b).
However, the underlying construction principle is the same.
}
\end{remark}

\begin{figure}[ht]
 \centering
 \subfigure{
 \includegraphics[width = .2\textwidth]{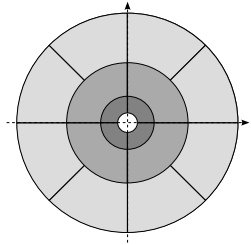}
 }
 \caption{Partition of the Fourier domain induced by second generation curvelets.}
 \label{fig:secgen}
 \end{figure}

\subsubsection{Ridgelets}

The earliest version of the ridgelet transform was introduced by Cand\`es~\cite{Can98} in 1998. It uses a univariate wavelet $\phi\in L^2(\R)$ to
map a function $f\in L^2(\R^d)$ to its transform coefficients
\[
\langle f , \sqrt{s} \phi (\langle s\nu, \cdot \rangle -t), \quad \nu\in\mathbb{S}^{d-1}, t\in\R, s\in\R_+.
\]
The function $x\mapsto  \sqrt{s} \phi (\langle s\nu, x \rangle -t)$ is a \emph{ridge function} (hence the name ridgelet) which only varies in the direction $\nu$.
Unfortunately, since this function is not in $L^2(\R^d)$, the definition, as it stands, does not make sense for every $f\in L^2(\R^d)$.
Similar to the continuous Fourier transform, however, the continuous version of this transform can be well-defined \cite{Can98}.

In order to avoid the problems associated with the lack of integrability of ridge functions, Donoho~\cite{Don2000} slightly relaxed the
definition of a ridgelet, allowing them a slow decay in the other directions. In the spirit of this more general approach, as pointed out
by Grohs \cite{GrohsRidLT}, one might define a ridgelet system as a system of functions of the form
\[
x\mapsto\sqrt{s} \psi(D_sR_\nu x-t),
\]
obtained by applying dilations $D_s = \mbox{diag}(s,1,\dots , 1)\in \R^{d\times d}$
for $s\in\R_+$ \pg{and rotations $R_\nu$, $\nu\in \mathbb{S}^{d-1}$ to some generator $\psi\in L^2(\R^d)$, which needs to be oscillatory
in one coordinate direction. The resulting system can again be shown to form a (tight) frame.} This more general definition can be stated \ms{in the case $d=2$} as follows.

\begin{definition}
\label{defi:ridgelets}
The frame $C_0(W^{(0)},W^{(1)},V)$ from Definition \ref{defi:curvelets} is called {\em ridgelet system}.
\end{definition}

The associated partition of Fourier domain is pictured in Figure~\ref{fig:alphacurve}(c).

\subsubsection{Shearlets} \label{subsec:shearlets}

Shearlets were introduced in \cite{GKL05}. The basic idea is to obtain a directional representation system from a fixed function by applying
shears, translations and parabolic dilations. The choice of shears instead of rotations for the change of orientation makes shearlets more adapted
to a digital grid than curvelets, thereby enabling faithful implementations. To allow a more uniform treatment of the different directions,
usually two generators with orthogonal orientations are used. Moreover, a distinct generator is utilized for the coarse-scale elements. Such
shearlet systems are called cone-adapted, since one can picture the Fourier plane as divided into a horizontal and a vertical cone, as well
as a coarse-scale box, associated with the respective generators. This as well as a typical Fourier domain tiling induced by a cone-adapted shearlet system can
be viewed in Figure~\ref{fig:shear}. For more information on shearlets, we refer to the survey chapter \cite{KL2012}.


\begin{figure}[ht]
 \centering
 \subfigure{
 \includegraphics[width = .2\textwidth]{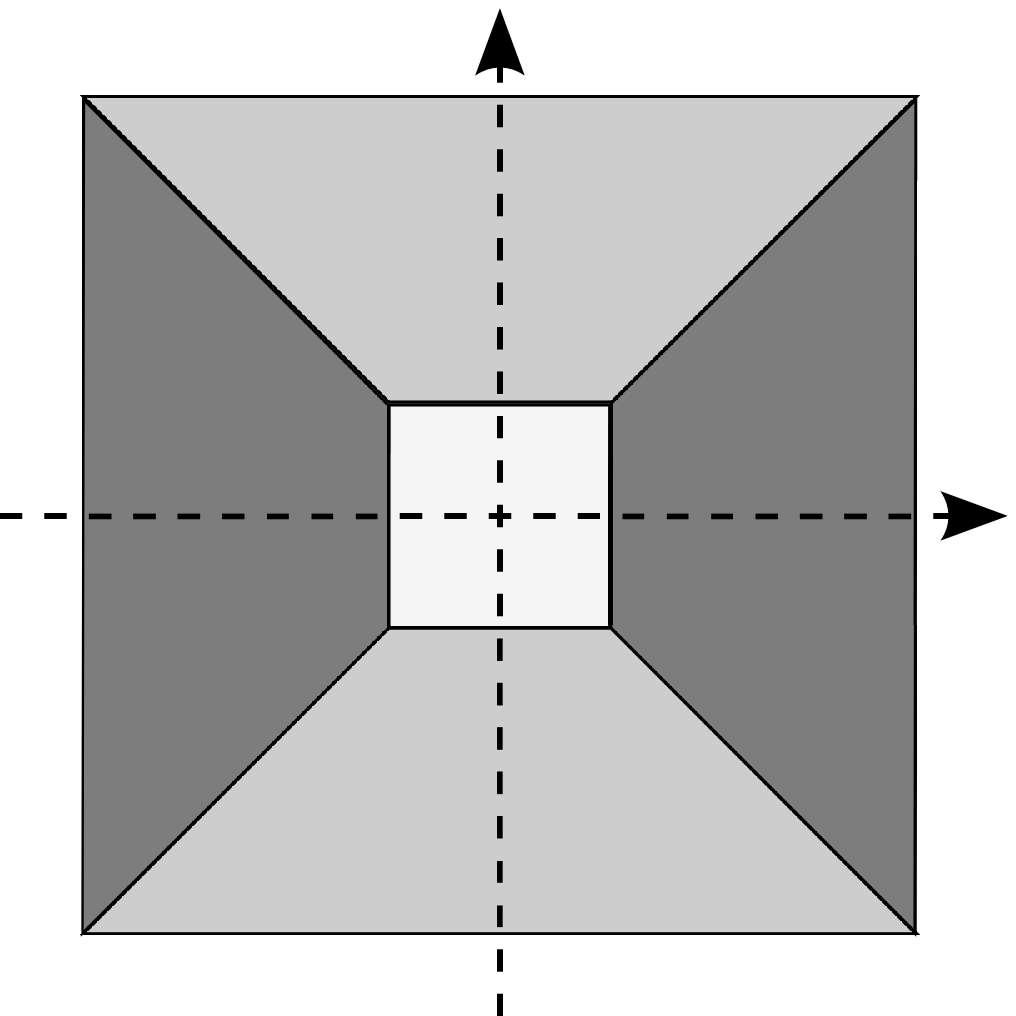}
 \put(-53,-17){(a)}
 }
 \qquad\qquad
 \subfigure{
 \includegraphics[width = .2\textwidth]{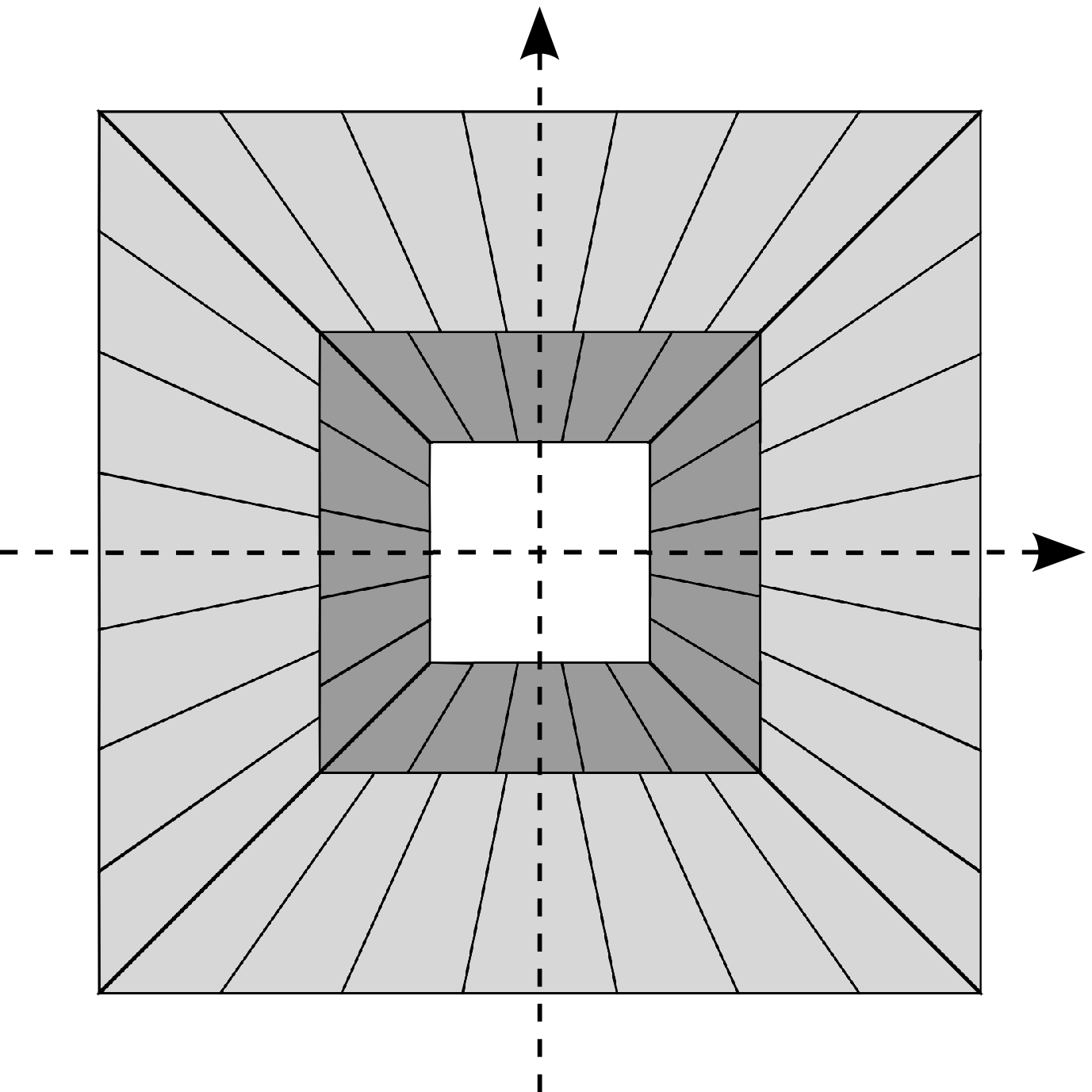}
 \put(-53,-17){(b)}
 }
 \caption{(a): The Fourier domain is partitioned into a horizontal and vertical double cone and a low-frequency box.
 (b): Partition of the Fourier domain induced by a cone-adapted shearlet system.}
 \label{fig:shear}
 \end{figure}

\ms{
For the definition of cone-adapted shearlets, we need
in addition to the scaling matrix $A_{\alpha,s}$ from \eqref{eq:scalingmatrix} its rotated version
\begin{equation}\label{eq:dilation}
\tilde{A}_{\alpha,s}=\begin{pmatrix} s^{\alpha} & 0 \\ 0 & s \end{pmatrix}, \quad s>0,
\end{equation}
as well as the shear matrix
\begin{equation}\label{eq:shearing}
S_{h}=\begin{pmatrix} 1 & h \\ 0 & 1 \end{pmatrix}, \quad h\in\R.
\end{equation} }


The cone-adapted (discrete) shearlet system is then defined as follows.

\begin{definition}\label{def:shearsys}
For $c\in\R_+$, the \emph{cone-adapted shearlet system}
$SH\big(\phi,\psi,\tilde{\psi};c\big)$ generated by $\phi,\psi,\tilde{\psi}\in L^2(\R^2)$ is defined by
\[
SH\big( \phi,\psi,\tilde{\psi};c \big)= \Phi(\phi;c) \cup \Psi(\psi;c) \cup \tilde{\Psi}(\tilde{\psi};c),
\]
where
\begin{align*}
 &\Phi(\phi;c)=\{\phi_k=\phi(\cdot - k) : k\in c\Z^2 \}, \\
 &\Psi(\psi;c)= \big\{\psi_{j,\ell,k}= 2^{3j/4}\psi(S_\ell A_{\frac12,2^{j}}\cdot-k ): j\ge0, |\ell|\le \lceil 2^{j/2} \rceil, k\in c\Z^2  \big\}, \\
 &\tilde{\Psi}(\tilde{\psi};c)
 = \big\{\tilde{\psi}_{j,\ell,k}= 2^{3j/4}\tilde{\psi}(S^T_\ell \tilde{A}_{\frac12,2^{j}}\cdot-k ): j\ge0, |\ell|\le \lceil 2^{j/2} \rceil, k\in c\Z^2  \big\}.
\end{align*}
\end{definition}

\begin{remark}
\ms{Utilizing two parameters $(c_1,c_2)\in\R^2_+$ instead of $c\in\R_+$ would allow more flexible rectangular sampling grids \cite{Kittipoom2010}.
For simplicity of notation, we chose to restrict our considerations to equal sampling in both spatial directions, i.e., a square sampling grid.
We want to remark however, that without much additional effort it is possible to also include the more general case in the discussion.
}
\end{remark}

\subsection{Definition of $\alpha$-Molecules}
\label{subsec:defalpha}

Aiming for a framework which encompasses the previously introduced multiscale systems, we first realize that their
parameter sets differ significantly. Thus a common parameter space has to be selected. Whereas wavelets only depend
on scale and position, ridgelets, curvelets as well as shearlets are all based on scale, orientation, and position.
Hence it seems appropriate to choose the common parameter space as a phase space with an additional scale parameter.

\begin{definition}
We define the \emph{parameter space} $\mathbb{P}$ by
\[
\mathbb{P}:= \mathbb{R}_+\times\mathbb{T}\times \mathbb{R}^2,
\]
where $\R_+=(0,\infty)$ and $\mathbb{T}=[-\frac{\pi}{2},\frac{\pi}{2}]$ denotes the torus with endpoints identified.
\end{definition}

Thus a point $p = (s,\theta,x)$ in the parameter space $\mathbb{P}$ describes a scale $s\in\R_+$, an orientation
$\theta\in\mathbb{T}$, and a location $x\in\R^2$.

To allow arbitrary index sets --  the necessity being discussed above -- we require mappings of those into the
just defined common parameter space $\mathbb{P}$. This leads us to the following definition.

\begin{definition}
A \emph{parametrization} consists of a pair $(\Lambda,\Phi_\Lambda)$, where $\Lambda$ is an index set and $\Phi_\Lambda$ is
a mapping
\[
\Phi_\Lambda:\left\{\begin{array}{ccc}\Lambda &\to & \mathbb{P},\\
\lambda \in\Lambda & \mapsto & \left(s_\lambda , \theta_\lambda , x_\lambda\right),
\end{array}\right.
\]
which associates with each $\lambda\in \Lambda$ a \emph{scale} $s_\lambda\in\R_+$, a \emph{direction}
$\theta_\lambda\in\mathbb{T}$, and a \emph{location} $x_\lambda\in\R^2$.
\end{definition}

Similar to all multiscale systems in applied harmonic analysis, also $\alpha$-molecules should \msch{follow the construction principle of
applying} certain operators to generating functions. ($\alpha-$)Scaling and translation operators are an obvious choice.
As an operator associated with the orientation index, two possibilities stand at attention, namely rotation and
shearing. In preference \msch{of} a more convenient choice -- recall that the shearing operator required us to utilize
two different generators in Subsection \ref{subsec:shearlets} -- and since we merely seek to introduce a
theoretical framework, we choose the rotation operator. Intriguingly, shearlets are still included in the framework
of $\alpha$-molecules as we will prove later, thereby showing its generality.

Our next decision concerns the generating functions. To ensure maximal flexibility, we allow those to change with
{\it each} index $\lambda \in \Lambda$, i.e., we employ a family of variable generators $(g^{(\lambda)})_\lambda \subseteq L^2(\R^2)$.
Certainly, to derive \msch{a} meaningful family, the generators have to satisfy certain time-frequency localization properties,
which are governed by a set of {\it control parameters}. Those are chosen as a quadruple $(L,M,N_1,N_2)$, where
$L$ describes the spatial localization, $M$ the number of directional (almost) vanishing moments, and $N_1,N_2$ the smoothness
of the generators.

After this preparation, we are now ready to face the definition of $\alpha$-molecules. For this, recall the notions
$A_{\alpha,s}$ for $\alpha$-scaling and $R_\theta$ for rotation from \eqref{eq:scalingmatrix} and \eqref{eq:rotation},
respectively. Also, we will \msch{use} the so-called analyst's brackets $\langle x \rangle = (1+x^2)^{\frac{1}{2}}$, $x\in\R$.
In this section as well as in the sequel, the notation $a\lesssim b$ shall indicate that the entities $a,b$, possibly depending on
some context dependent parameters, satisfy $a\le C\cdot b$ for a positive constant $C>0$, which is independent of the parameters.
If both $a\lesssim b$ and $b\lesssim a$, we denote this by $a\asymp b$.

\begin{definition}\label{defi:alphamolecules}
Let $\alpha \in [0,1]$, \ms{let $L,M,N_1,N_2 \in \N_0 \cup \{\infty\}$}, and let $(\Lambda,\Phi_\Lambda)$ be a parametrization.
A family $(m_\lambda)_{\lambda \in \Lambda}$ \ms{of functions $m_\lambda\in L^2(\R^2)$} is called a \emph{system of $\alpha$-molecules with respect to the parametrization
$(\Lambda,\Phi_\Lambda)$ of order $(L,M,N_1,N_2)$}, if it can be written as
\[
m_\lambda (x) = s_\lambda^{(1+\alpha)/2} g^{(\lambda)} \left(A_{\alpha,s_\lambda}R_{\theta_\lambda}\left(x - x_\lambda\right)\right)
\]
such that, for all $|\rho|\le L$,
\begin{equation}\label{eq:molcond}
\left| \partial^{\rho} \hat g^{(\lambda)}(\xi)\right| \lesssim \min\left\{1,s_\lambda^{-1} + |\xi_1| + s_\lambda^{-(1-\alpha)}|\xi_2|\right\}^M
\cdot \left\langle |\xi|\right\rangle^{-N_1} \cdot \langle \xi_2 \rangle^{-N_2}
\end{equation}
The implicit constants shall be uniform over $\lambda\in \Lambda$, and in case \pg{that} one or several control parameters equal
infinity, the respective quantity can be arbitrarily large.
\end{definition}

The condition on the generators $g^{(\lambda)}$ in \eqref{eq:molcond} ensures that the Fourier transforms of $\alpha$-molecules
have essential frequency support in a pair of opposite wedges associated to a certain orientation, and essential spatial support
in a rectangle with scale-dependent side lengths. This can \msch{perhaps be more conveniently deduced} from the corresponding version of
\eqref{eq:molcond} in polar coordinates, which can be easily computed to be
\begin{equation}\label{eq:moldecaypolar2}
    \left|\hat m_{\lambda}(\xi)\right|
    \lesssim  s_\lambda^{-(1+\alpha)/2} \cdot \min\left\{1, s_\lambda^{-1}( 1 + r )\right\}^M
    \cdot \left\langle \min \{s_\lambda^{-\alpha},s_\lambda^{-1}\} r\right \rangle^{-N_1}
    \cdot \langle s_\lambda^{-\alpha}r\sin(\varphi + \theta_\lambda)\rangle^{-N_2}.
\end{equation}

\begin{figure}[ht]
 \centering
 \subfigure{
 \includegraphics[width = .23\textwidth]{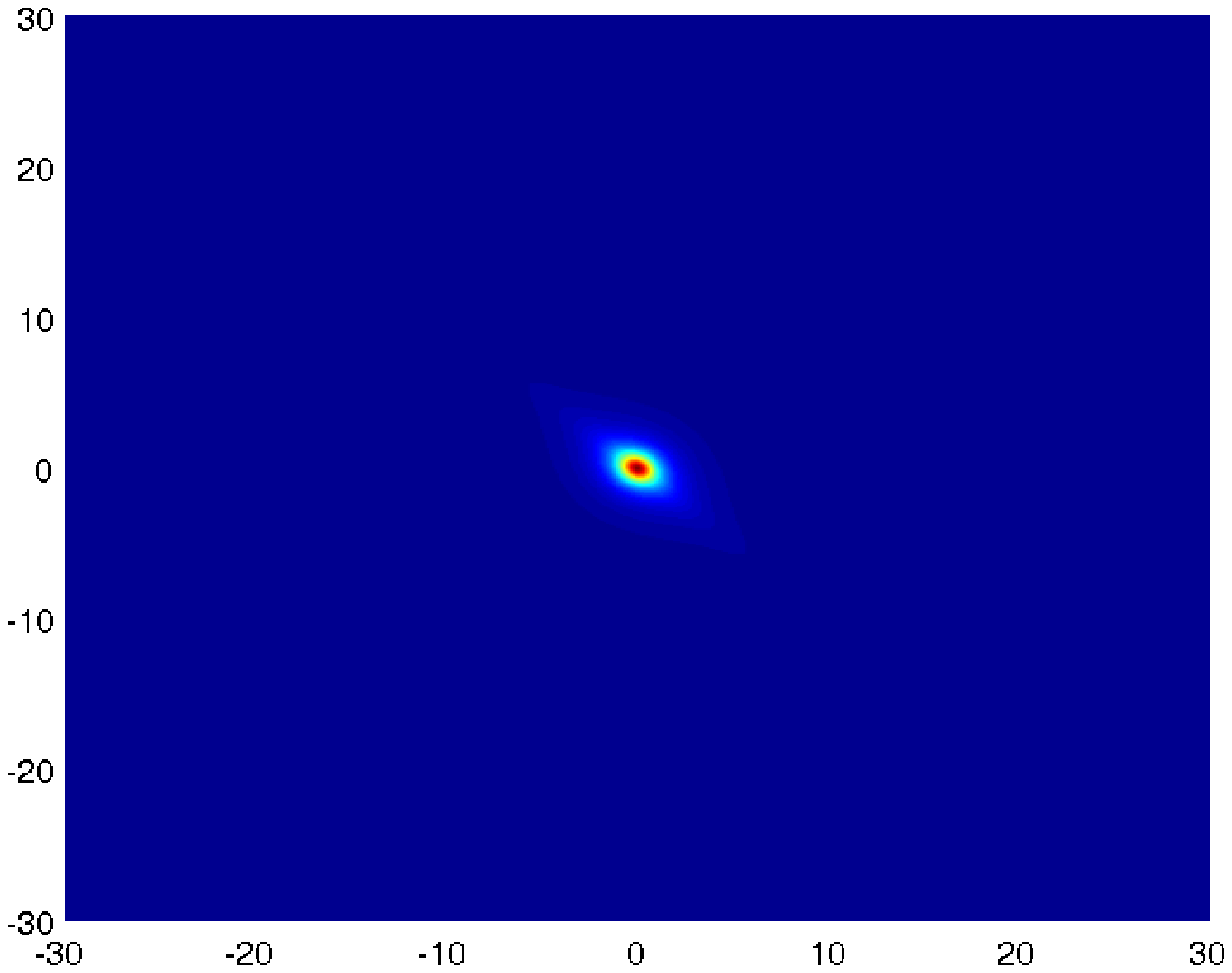}
 \put(-57,-17){(a)}
 }   
 \subfigure{
 \includegraphics[width = .23\textwidth]{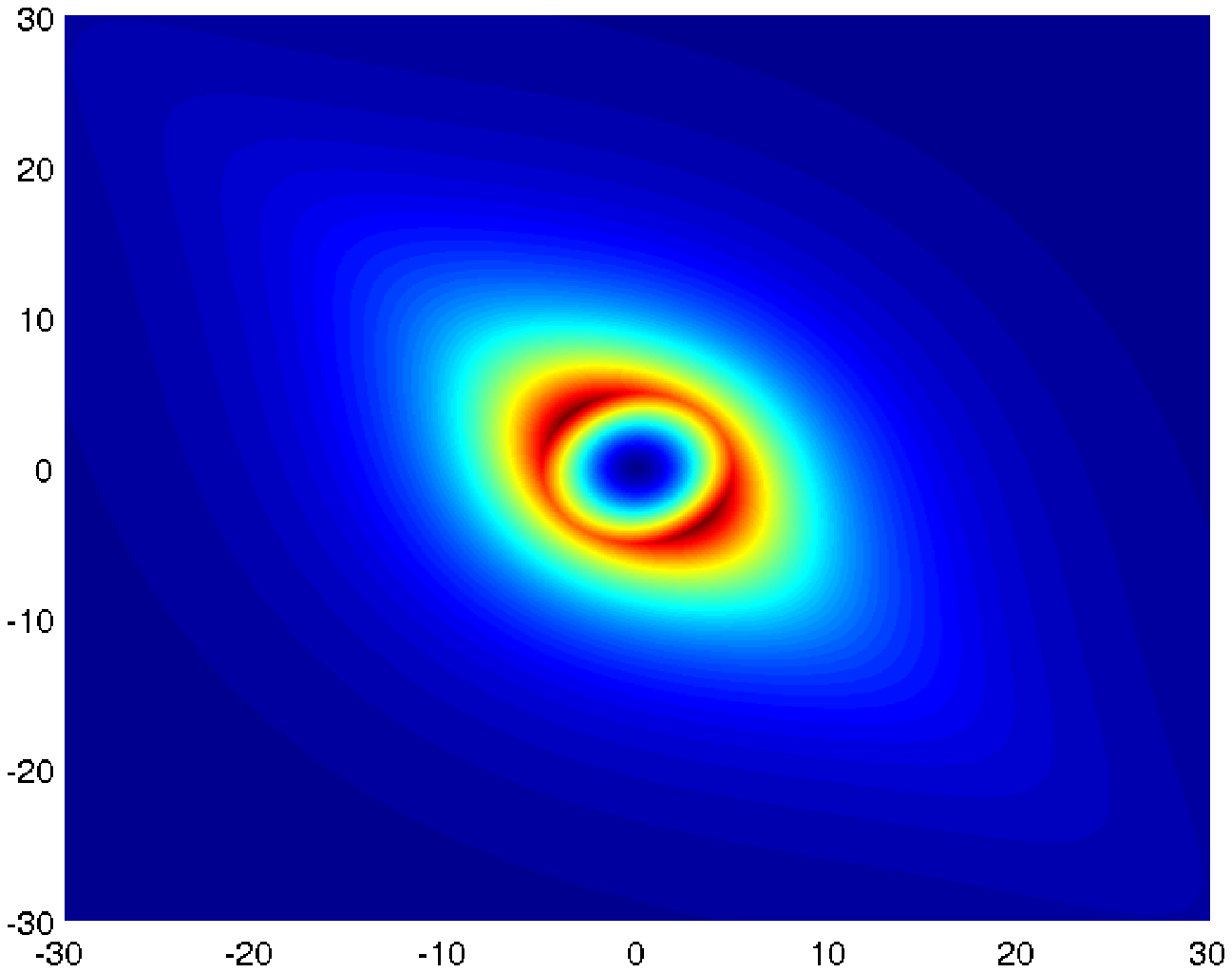}
 \put(-57,-17){(b)}
 }   
 \subfigure{
 \includegraphics[width = .23\textwidth]{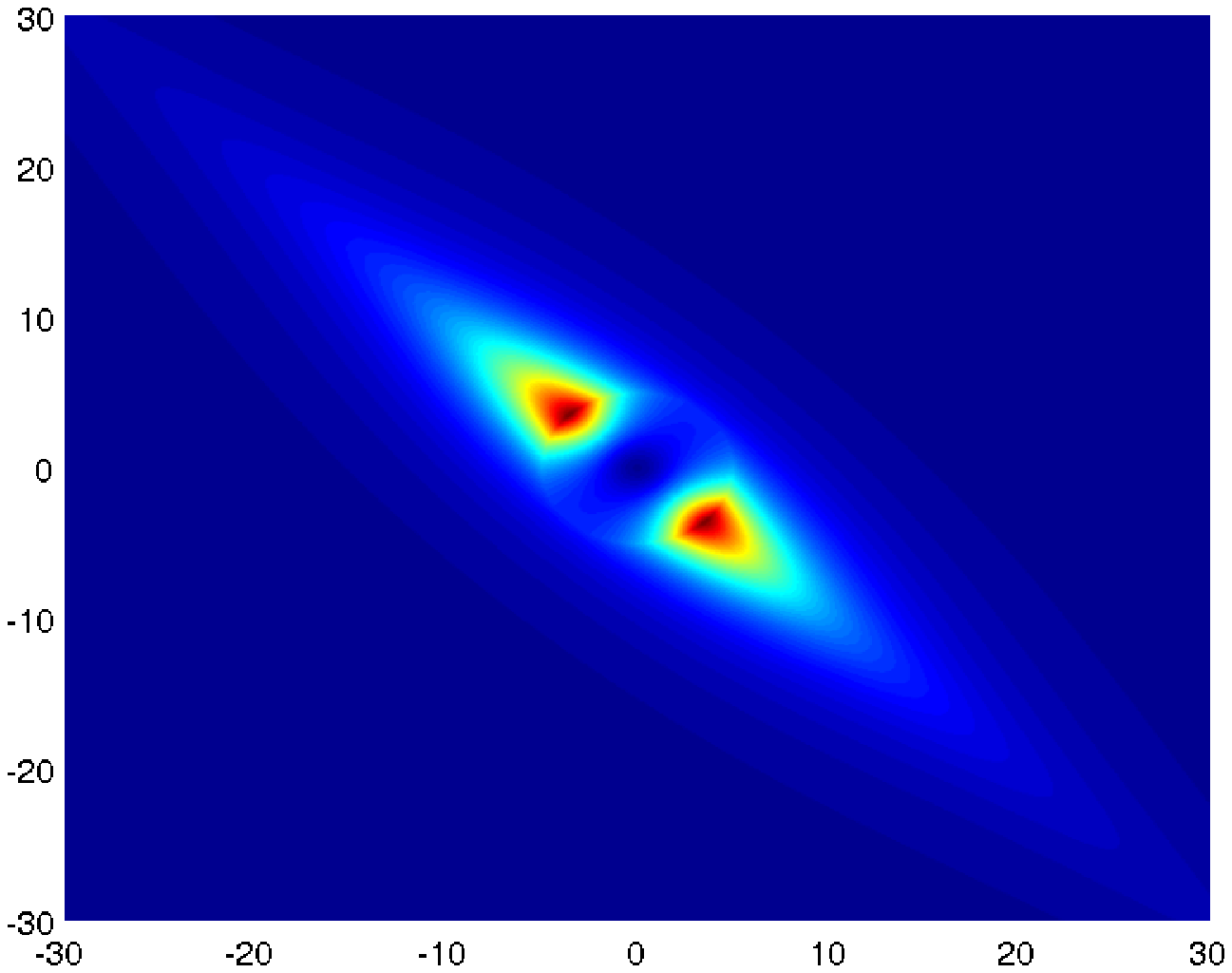}
 \put(-57,-17){(c)}
 }   
 \subfigure{
 \includegraphics[width = .23\textwidth]{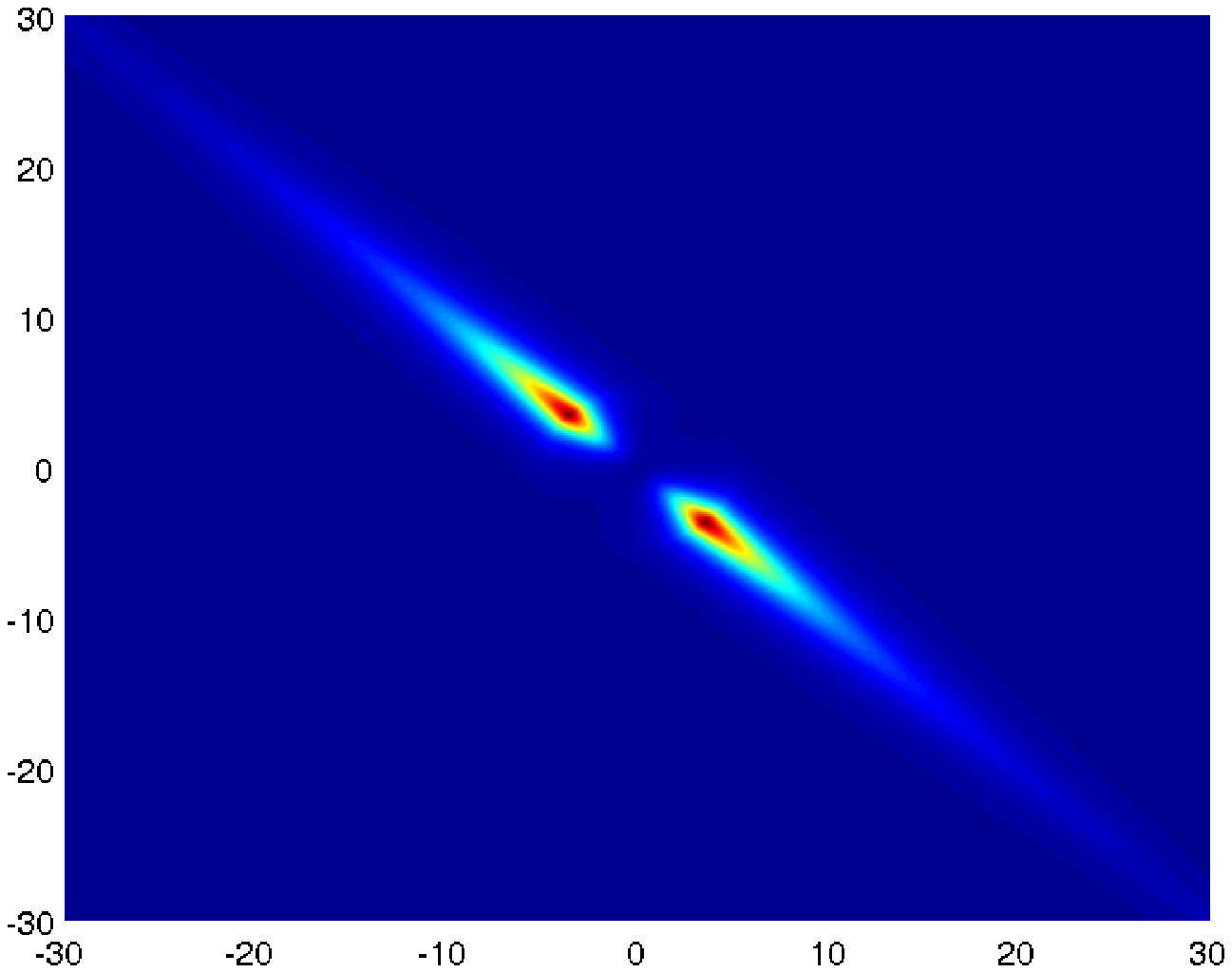}
 \put(-57,-17){(d)}
 }   
 \caption{\ms{Frequency support of $\alpha$-molecules ($N_1=2$, $N_2=1$, $M=3$, $\theta=\pi/4$)
  with (a): $s = 1$ and independent of $\alpha$, (b): $s=6$ and $\alpha=1$, (c): $s=6$ and $\alpha=1/2$, and
  (d): $s=6$ and $\alpha=0$.}}
 \label{fig:freqsuppalphamol}
 \end{figure}

To illustrate this behavior, several possibilities for such $\alpha$-molecules are shown in Figure~\ref{fig:freqsuppalphamol}, also
demonstrating the inclusion of different \pg{anisotropies} as well as different partitions of Fourier domain. The reader might want
to compare those with the partitions given by wavelets (Figure~\ref{fig:wave} and Figure~\ref{fig:alphacurve}(a)), curvelets (Figure~\ref{fig:alphacurve} and Figure~\ref{fig:secgen}),
ridgelets (Figure~\ref{fig:alphacurve}(c)), and shearlets (Figure~\ref{fig:shear}). These figures in fact already visually indicate
that those systems as well as a variety of novel partitions of Fourier domain are included such as, for instance, the partition illustrated
in \pg{Figure \ref{fig:novel}}.

\begin{figure}[ht]
 \centering
 \qquad\qquad
 \subfigure{
 \includegraphics[width = .2\textwidth]{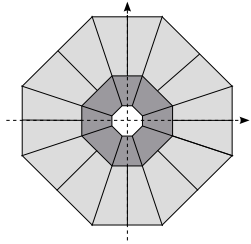}
 }
 \caption{Novel partition of the Fourier domain.}
 \label{fig:novel}
 \end{figure}

\section{Examples of $\alpha$-Molecules}
\label{sec:examples}

Having stated and discussed the novel notion of a system of $\alpha$-molecules, the immediate question arises whether the prominent
representation systems presented in Subsection \ref{subsec:prominent} are included, and if so, with respect to which parametrizations
$(\Lambda,\Phi_\Lambda)$ and of which orders $(L,M,N_1,N_2)$. For this investigation, we follow the same ordering as in Subsection \ref{subsec:prominent},
i.e., first wavelets, then curvelets, followed by ridgelets, and finally shearlets.

\subsection{Wavelets}

For this exposition, we focus on bandlimited wavelets with infinitely many vanishing moments. Therefore, we additionally assume that
the functions $\phi^0,\phi^1\in L^2(\R)$ used for the construction of $W(\phi^0,\phi^1;\sigma,\tau)$ satisfy
\begin{align}\label{eq:wavreg}
\hat{\phi}^0,\,\hat{\phi}^1\in C^L(\R) \qquad\text{ for some }L\in\N_0\cup\{\infty\},
\end{align}
and that there exist $0<a$ and $0<b<c$ such that
\begin{align}\label{eq:wavsupp}
\supp \hat{\phi}^0 \subset [-a,a]=:J^{(0)} \quad \text{and} \quad
\supp \hat{\phi}^1 \subset [-c,c]\backslash [-b,b]=:J^{(1)}.
\end{align}
These conditions are fulfilled, if, for instance, $\phi^0,\phi^1\in L^2(\R)$ are the generators of a Lemari\'{e}-Meyer wavelet
system.

The following result now shows that these wavelet systems are instances of $\alpha$-molecules of arbitrarily large order.

\begin{prop}
Let $\sigma>1$, $\tau>0$ be fixed, and assume that the functions $\phi^0$, $\phi^1$ satisfy the assumptions \eqref{eq:wavreg} and \eqref{eq:wavsupp}.
Then the wavelet system $W(\phi^0,\phi^1;\sigma,\tau)$ constitutes a system of $1$-molecules of order $(L,\infty,\infty,\infty)$ with respect
to the parametrization $(\Lambda^w,\Phi^w)$ given by
\[
\Lambda^w=\big\{ ((0,0),0,k) ~:~ k\in\Z^2 \big\} \cup \big\{ (e,j,k) ~:~ e\in E\backslash\{(0,0)\},\, j\in\N_0,\, k\in\Z^2 \big\}
\]
and
\[
\Phi^w : \Lambda^w\rightarrow\mathbb{P},\quad (e,j,k)\mapsto (\sigma^j,0,\tau \sigma^{-j}k).
\]
\end{prop}

\begin{proof}
For $(e,j,k)\in\Lambda^w$ we define the functions $g^{(e,j,k)}:=\psi^e$, with $\psi^e$ being the functions given in \eqref{eq:wavegen}. Since $\hat{g}^{(e,j,k)}=\hat{\psi^e}$, we have
$\hat{g}^{(e,j,k)}\in C^L(\R^2)$ by \eqref{eq:wavreg}. Further, \eqref{eq:wavsupp} implies that
\begin{align*}
\supp \widehat{\psi}^{e} \subset S^{e}:=J^{(e_1)}\times J^{(e_2)} \quad\text{ for all }e\in E.
\end{align*}
Hence $\supp (\partial^{\rho}\hat{g}^{(e,j,k)})\subset S^e$ for every $|\rho|\le L$ and for all $(e,j,k)\in\Lambda^w$. This proves that
the functions $g^{(e,j,k)}$ satisfy condition \eqref{eq:molcond}. Since the wavelets can be written in the form
\[
\psi^e_{j,k}:=\sigma^j\psi^e(\sigma^j (x - \tau \sigma^{-j}k))=\sigma^j g^{(e,j,k)}(\sigma^j (x - \tau \sigma^{-j}k)),
\]
the proof is finished.
\end{proof}

We remark that the framework of $\alpha$-molecules can be shown to also comprise other constructions such as systems of
compactly supported wavelets or bandlimited radial wavelets.

\subsection{Curvelets}
\label{subsec:curveletsmol}

In Subsection \ref{subsec:curvelets}, we introduced $\alpha$-curvelets, which are a generalization of second generation
curvelets to different types of scalings. In \cite{CD02} the authors introduced the notion of curvelet molecules, which
are closely related to curvelets. To also include those in our consideration -- which will turn out to be beneficial
later --, we start by introducing yet a further extension, which we coin $\alpha$-curvelet molecules.

Interestingly, we can employ the general framework of $\alpha$-molecules for this, defining $\alpha$-curvelet molecules as those
systems with a particular parametrization.  Those will then be shown to encompass both $\alpha$-curvelets and curvelet
molecules, which immediately implies that those systems are in fact instances of $\alpha$-molecules.

\begin{definition}
\label{defi:alphacurveletpara}
Let $\alpha\in[0,1]$ and $\tau>0$, $\sigma>1$ be some fixed
parameters. Further, let $(\omega_j)_{j\in\N_0}$ be a sequence of positive real numbers with $\omega_j\asymp \sigma^{-j(1-\alpha)}$.
An \emph{$\alpha$-curvelet parametrization $(\Lambda^c,\Phi^c)$} is given by an index set $\Lambda^c$ of the form
\[
\Lambda^c:=\left\{ (j,\ell,k) ~:~ j\in\N_0,\ \ell\in\Z \text{ with } |\ell|\le L_j \text{ for some } L_j\in\N_0\cup\{\infty\},\ k\in\Z^2\right\},
\]
and a mapping $\Phi^c$ defined by
\[
\Phi^c : \Lambda^c \rightarrow \mathbb{P}, \quad (j,\ell,k) \mapsto (s_\lambda,\theta_\lambda,x_\lambda) := (\sigma^j,\ell\omega_j,\tau R^{-1}_{\ell\omega_j}A^{-1}_{\alpha,\sigma^j}k).
\]
A \emph{family of $\alpha$-curvelet molecules} is a family of $\alpha$-molecules with respect to an $\alpha$-curvelet parametrization.
\end{definition}

Notice that the parameters $\sigma>1$ and $\tau>0$ are sampling constants, which determine the fineness of the sampling grid,
$\sigma$ for the scale parameters and $\tau$ for the space parameters. The values $(\omega_j)_{j\in\N_0}$ prescribe the step
size of the angular sampling at each scale $j\in\N_0$.

\begin{prop}\label{prop:acurvmol}
The following statements hold.
\begin{enumerate}
\item[(i)] \ms{Curvelet molecules of regularity $R\in\N_0$}, \pg{as defined in \cite{CD02}}, are $\frac12$-curvelet molecules of order $(\infty,\infty,R/2,R/2)$.
\item[(ii)] \ms{Second generation curvelets are $\frac12$-curvelet molecules of order $(\infty,\infty,\infty,\infty)$ with parameters $\sigma=2$ and $\tau=1$.}
\item[(iii)] For each $\alpha\in[0,1]$, the $\alpha$-curvelet frame $C_\alpha(W^{(0)},W,V)$ is a system of $\alpha$-curvelet
molecules of order $(\infty,\infty,\infty,\infty)$ with parameters $\sigma=2$ and $\tau=1$.
\end{enumerate}
\end{prop}

\begin{proof}
(i) and (ii) \ms{were proved in \cite{Grohs2011}.}

(ii) Due to rotation invariance, it suffices to show that the generators
\[
g_{j,0,0}:=2^{-j(1+\alpha)}\psi_{j,0,0} (A^{-1}_{\alpha,2^j} \cdot), \quad j\in\N_0,
\]
satisfy \eqref{eq:molcond}. On the Fourier side they take the form
\[
\hat{g}_{j,0,0}=\hat{\psi}_{j,0,0} (A_{\alpha,2^j} \cdot).
\]
From $\supp \hat{\psi}_{0,0,0} \subset [-\frac{1}{2},\frac{1}{2}]^2=:\Xi_0$ and
\[
\supp \hat{\psi}_{j,0,0} \subset [-2^{j-1},2^{j-1}]\times[-2^{j\alpha-1},2^{j\alpha-1}], \quad j\in\N,
\]
it follows that
\[
\supp \hat{g}_{j,0,0} \subset \Xi_{0} \quad \mbox{for all } j\in\N_0.
\]
Next, for $j\in\N$ we observe that the functions $\hat{\psi}_{j,0,0}$ vanish on the squares $[-2^{j-7},2^{j-7}]^2$,
which implies that $\hat{g}_{j,0,0}$ is equal to zero on $[-2^{-7},2^{-7}]^2$ if $j\in\N$.

Clearly, we have $g\in C^\infty(\R^2)$ and the derivatives $\partial^{\rho} g$ are subject to the same support conditions
as the function $g$. Thus, condition \eqref{eq:molcond} follows for arbitrary order $(L,M,N_1,N_2)$.
\end{proof}

We obtain immediately the following corollary.

\begin{cor}\label{coro:curveletmol}
For each $\alpha\in[0,1]$, the $\alpha$-curvelet frame $C_\alpha(W^{(0)},W,V)$ is a system of $\alpha$-molecules
of order $(\infty,\infty,\infty,\infty)$ with respect to the parametrization $(\Lambda^c,\Phi^c)$.
\end{cor}

\subsection{Ridgelets}

The ridgelet frame $C_0(W^{(0)},W^{(1)},V)$ (cf. Definition \ref{defi:ridgelets}) is a special case of $\alpha$-molecules
as a direct consequence of Corollary \ref{coro:curveletmol}.

\begin{prop}
The ridgelet frame $C_0(W^{(0)},W^{(1)},V)$ is a system of $\,0$-molecules of order $(\infty,\infty,\infty,\infty)$
with respect to the parametrization $(\Lambda^c,\Phi^c)$.
\end{prop}

One might go even one step further, and -- based on considerations and extensions undertaken in \cite{GrohsRidLT}~--
also introduce a system of ridgelet molecules by the following definition.

\begin{definition}
A system of $0$-curvelet molecules is called a system of \emph{ridgelet molecules}.
\end{definition}

Thus, with Proposition \ref{prop:acurvmol}, also ridgelet molecules are immediately instances of $\alpha$-molecules.

\subsection{Shearlets}
\label{subsec:shearletmolecules}

Based on the definition of cone-adapted shearlet systems as stated in Definition \ref{def:shearsys}, two extensions
can be witnessed in the \pg{literature}: shearlet molecules \cite{GL08} with a subsequent generalization in \cite{Grohs2011}
as well as $\alpha$-shearlets (also called hybrid shearlets) in \cite{Kutyniok2012,Kei13}. Thus, in a similar fashion
as in the curvelet case (cf. Subsection \ref{subsec:curveletsmol}), we will introduce $\alpha$-shearlet molecules
and first prove that they are indeed instances of $\alpha$-molecules. This is significantly more difficult
than for curvelets due to the form of the parametrization which arises from the utilization of shearing instead of
rotation. This result can then be used to analyze shearlet molecules in the sense of \cite{GL08} and $\alpha$-shearlets with regard to their
membership in the framework of $\alpha$-molecules.

\ms{
For the definition of $\alpha$-shearlet molecules, it is convenient to resort to the following notation.
Recalling \eqref{eq:scalingmatrix}, \eqref{eq:dilation}, and \eqref{eq:shearing},
we put $A^0_{\alpha,s}:=A_{\alpha,s}=\mbox{diag}(s,s^\alpha)$ and $A^1_{\alpha,s}:=\tilde{A}_{\alpha,s}=\mbox{diag}(s^\alpha,s)$ for the scaling
matrices, and denote the
shearing matrices by $S^0_{\ell,j} := S_{\ell \eta_j}$ and
$S_{\ell,j}^1 := S_{\ell \eta_j}^T$.  

\begin{definition}
\label{defi:alphashearletmol}
Let $\alpha\in[0,1]$ and $\tau>0$, $\sigma>1$ be some fixed
parameters. Further, let $(\eta_j)_{j\in\N_0}$ be a sequence of positive real numbers with $\eta_j\asymp \sigma^{-j(1-\alpha)}$
and put $\eta_{-1}=0$.
We define the index set
\begin{equation}\label{eq:shearindex}
\Lambda^s := \Lambda_0^s \cup \left\{ (\varepsilon, j,\ell,k) ~:~ \varepsilon \in \{0,1\},\
     j\in\N_0,\ \ell\in\Z \text{ with } |\ell|\le L_j, \text{ and }L_j\lesssim \sigma^{j(1-\alpha)},\, k\in\Z^2 \right\}
\end{equation}
with $\Lambda_0^s := \left\{ (0,-1,0,k) ~:~ k\in\Z^2 \right\}$ and
call a system $\Sigma:=\{\psi_\lambda: \lambda\in \Lambda^s\}$ defined by
\[
\psi_{(\varepsilon , j, \ell , k)}(\cdot) := \sigma^{(1+\alpha)j/2}\gamma^\varepsilon_{j,\ell,k} \left(A^\varepsilon_{\alpha,\sigma^j} S_{\ell,j}^\varepsilon \cdot - \tau k\right)
\quad \mbox{for some } \gamma^\varepsilon_{j,\ell,k} \in L^2(\mathbb{R}^2)
\]
a system of \emph{$\alpha$-shearlet molecules} of order $(L,M,N_1,N_2)$, if, for every $\rho\in \mathbb{N}_0^2$ with $|\rho|\le L$,
\begin{equation}\label{eq:shearcond}
|\partial^\rho \hat \gamma^\varepsilon_{j,\ell,k}(\xi_1,\xi_2)|\lesssim \min \left\{1,\sigma^{-j}+|\xi_{1+\varepsilon}|+\sigma^{-j(1-\alpha)}|\xi_{2-\varepsilon}|\right\}^M
\cdot \langle |\xi |\rangle^{-N_1} \cdot \langle \xi_{2-\varepsilon}\rangle^{-N_2}
\end{equation}
with an implicit constant independent of the indices $(\varepsilon,j,\ell,k)\in\Lambda^s$.
\end{definition}

Notice that the indices $\Lambda_0^s$ at scale $j=-1$ correspond to the coarse scale elements.

\medskip

We will next see, that although $\alpha$-shearlet molecules are based on shearing rather than rotation,
they are still instances of $\alpha$-molecules. For this, we utilize a special parametrization.

\begin{definition}
\label{defi:alphashearletpara}
With parameters given as in Definition~\ref{defi:alphashearletmol}, an \emph{$\alpha$-shearlet parametrization $(\Lambda^s,\Phi^s)$} consists of
an index set $\Lambda^s$ of the form \eqref{eq:shearindex} together with a mapping $\Phi^s$ defined by
\[
\Phi^s : \Lambda^s \rightarrow \mathbb{P}, \quad (\varepsilon,j,\ell,k) \mapsto (s_\lambda,\theta_\lambda,x_\lambda):=
\left(\sigma^j,\varepsilon\pi/2+\arctan(-\ell \eta_j),\left(S_{\ell,j}^\varepsilon\right)^{-1} A^\varepsilon_{\alpha,\sigma^{-j}}k\right).
\]
\end{definition}

Now we are ready to state the essential result, that $\alpha$-shearlet molecules are indeed $\alpha$-molecules. Since the proof is rather long and technical,
we outsource it to Subsection \ref{subsec:proofshearletmol}. }

\begin{prop}\label{prop:shearletmol}
A system of $\alpha$-shearlet molecules of order $(L,M,N_1,N_2)$ constitutes a system of $\alpha$-molecules of
the same order with respect to the associated $\alpha$-shearlet parametrization.
\end{prop}

We now return to the question of whether shearlet molecules in the sense of \cite{GL08} and $\alpha$-shearlets are instances of $\alpha$-molecules.
For this, we first recall the definition of $\alpha$-shearlets, which can be regarded as a version of Definition
\ref{def:shearsys} with flexible scaling, thereby providing like $\alpha$-curvelets a parametrized family of systems ranging from wavelets
to ridgelets. To not confuse this parameter with the parameter $\alpha$ from $\alpha$-molecules, we rename it $\beta$.

\begin{definition}\label{def:betashearsystem}
For $c\in\R_+$ and $\beta\in(1,\infty)$, the \emph{cone-adapted $\beta$-shearlet system} $SH\big(\phi,\psi,\tilde{\psi};c,\beta\big)$
generated by $\phi,\psi,\tilde{\psi}\in L^2(\R^2)$ is defined by
\[
SH\big( \phi,\psi,\tilde{\psi};c,\beta \big)= \Phi(\phi;c,\beta) \cup \Psi(\psi;c,\beta) \cup \tilde{\Psi}(\tilde{\psi};c,\beta),
\]
where
\begin{align*}
&\Phi(\phi;c,\beta)=\{\phi_k=\phi(\cdot - k) : k\in c\Z^2 \}, \\
&\Psi(\psi;c,\beta)= \big\{\psi_{j,\ell,k}= 2^{j(\beta+1)/4}\psi(S_\ell A_{\beta^{-1},2^{j\beta/2}}\cdot-k ): j\ge0, |\ell|\le \lceil 2^{j(\beta-1)/2} \rceil, k\in c\Z^2  \big\}, \\
&\tilde{\Psi}(\tilde{\psi};c,\beta)
 = \big\{\tilde{\psi}_{j,\ell,k}= 2^{j(\beta+1)/4}\tilde{\psi}(S^T_\ell \tilde{A}_{\beta^{-1},2^{j\beta/2}}\cdot-k ): j\ge0, |\ell|\le \lceil 2^{j(\beta-1)/2} \rceil, k\in c\Z^2  \big\}.
\end{align*}
\end{definition}

\ms{
The following result shows that shearlet molecules as well as cone-adapted $\beta$-shearlet systems -- with either
band-limited or compactly supported generators -- are instances of $\alpha$-molecules.

In the band-limited case, we require the generators $\phi,\psi,\tilde{\psi} \in L^2(\R^2)$ to have support of the form
\[
\supp \phi\subset Q, \qquad \sup \psi\subset W, \qquad \sup \tilde{\psi}\subset \tilde{W},
\]
where $Q\subset\R^2$ is a cube centered at the origin and $W,\,\tilde{W}\subset\R^2$ satisfy
\[
W\subset [-a,a]\times ([-c,-b]\cup[b,c]),\qquad \tilde{W}\subset ([-c,-b]\cup[b,c])\times[-a,a]
\]
for some $0<b<c$ and $0<a$.

In the compact case, the coarse-scale generator $\phi$ shall satisfy
 \[
 \phi\in C_0^{N_1+N_2}(\R^2).
 \]
 Furthermore, we assume the separability of $\psi\in L^2(\R^2)$, i.e.
 $
 \psi(x_1,x_2)=\psi_1(x_1)\psi_2(x_2),
 $
 and let $\tilde{\psi}$ be its rotation by $\pi/2$. Finally, the functions $\psi_1,\,\psi_2$ shall satisfy
 \[
 \psi_1\in C_0^{N_1}(\R) \quad \mbox{and} \quad  \psi_2\in C_0^{N_1+N_2}(\R),
 \]
 and for $\psi_1$ we assume $M\in\N_0$ vanishing moments.

 We wish to emphasize that there is a distinct difference
 between band-limited and compactly supported generators, as can also be read below from the different orders of the $\alpha$-molecules they induce.
}

\begin{prop}\label{prop:ashearmol}
The following statements hold.
\begin{enumerate}
\item[(i)] \ms{Shearlet molecules of regularity $R\in\N_0$,
\pg{as defined in \cite{GL08}}, are $\frac12$-molecules of order $(\infty,\infty,R/2,R/2)$.}
\item[(ii)] \ms{For each $\beta\in(1,\infty)$, $c\in\R_+$, and band-limited generators $\phi,\psi$, and $\tilde{\psi}$ subject to the conditions above, the
cone-adapted $\beta$-shearlet system $SH\big(\phi,\psi,\tilde{\psi};c,\beta\big)$ is a system of $\beta^{-1}$-molecules
of order $(\infty,\infty,\infty,\infty)$ with respect to the parametrization $(\Lambda^s,\Phi^s)$
with $\tau=c$, $\sigma=2^{\beta/2}$, $\eta_j=\sigma^{-j(1-\alpha)}$ and $L_j=\lceil \sigma^{j(1-\alpha)} \rceil$.}
\item[(iii)] \ms{For each $\beta\in(1,\infty)$, $c\in\R_+$, and compactly supported generators $\phi,\psi$, and $\tilde{\psi}$ subject to the conditions above, the
cone-adapted $\beta$-shearlet system $SH\big(\phi,\psi,\tilde{\psi};c,\beta\big)$ is a system of $\beta^{-1}$-molecules
of order $(L,M-L,N_1,N_2)$, where $L\in\{0,\ldots,M\}$ arbitrary, with respect to the parametrization
$(\Lambda^s,\Phi^s)$ with $\tau=c$, $\sigma=2^{\beta/2}$, $\eta_j=\sigma^{-j(1-\alpha)}$ and $L_j=\lceil \sigma^{j(1-\alpha)} \rceil$.}
\end{enumerate}
\end{prop}

Part (i) was proved in \cite{Grohs2011}. Part (ii) uses similar arguments as
the proof of Proposition \ref{prop:acurvmol}(ii). Thus the only interesting part is part (iii). Its proof shows that
those cone-adapted $\beta$-shearlet systems are in fact instances of $\alpha$-shearlet molecules, and thus by Proposition
\ref{prop:shearletmol} also instances of $\alpha$-molecules. Since this part is rather technical, we placed it in Subsection \ref{subsec:proofashearmol}.

Thus, even various versions of shearlet systems are united under the roof of $\alpha$-molecules. From the discussed
examples, this is maybe the most notable special \pg{case} due to the already mentioned difficulty with the seemingly not
consistent (shear-based) parametrization.

\section{Analysis of the Cross-Gramian}
\label{sec:gramian}

One main goal of the theory of $\alpha$-molecules is the unified treatment of sparse approximation properties of
multiscale systems within the area of applied harmonic analysis. Thus, it is crucial to be able to compare such
properties of two different systems. This in turn requires us to consider and analyze the cross-Gramian matrix
of two systems of $\alpha$-molecules.

We now see the benefit of having a common parameter space for all systems of $\alpha$-molecules. \msch{Utilizing parametrizations will enable
a comparison of different systems despite possibly incompatible index sets}. Still, we require a \msch{notion of distance on the parameter space}.
Recalling the definition $\mathbb{P}:= \mathbb{R}_+\times\mathbb{T}\times \mathbb{R}^2$,
we observe that the parameter space is a composition of a scaling space $\mathbb{R}_+$ and what is typically termed {\it phase space}
$\mathbb{T}\times \mathbb{R}^2$. For the phase space, a pseudodistance was introduced by Smith in \cite{Smith1998a}, \msch{which}
was \msch{later} tailored to curvelet analysis by Cand\`{e}s and Demanet in \cite{CD02}, \msch{who extended it to also include the scaling space}.
\msch{This scaled version was subsequently used} (with slight adaptions) in \cite{GL08} for shearlet molecules
and in \cite{Grohs2011} for parabolic molecules.

We though now require an $\alpha$-scaled version, which can be defined in the following way.

\begin{definition} \label{defi:indexdistance}
Let $\alpha \in [0,1]$, and let $(\Lambda,\Phi_\Lambda)$ and $(\Delta,\Phi_\Delta)$ be two parametrizations. We
then define the associated \emph{$\alpha$-scaled index distance} $\omega_\alpha : \Lambda \times \Delta \to \msch{[1,\infty)}$
as follows. For two indices $\lambda\in\Lambda$ and $\mu\in\Delta$ and associated images in $\mathbb{P}$ denoted by
\[
\left(s_\lambda , \theta_\lambda , x_\lambda\right) := \Phi_\Lambda(\lambda) \quad \mbox{and} \quad
\left(s_\mu , \theta_\mu , x_\mu\right) := \Phi_\Delta(\mu),
\]
we set
\[
\omega_\alpha \left(\lambda,\mu\right):= \max\Big\{\frac{s_\lambda}{s_\mu},\frac{s_\mu}{s_\lambda} \Big\} \left(1 + d_\alpha\left(\lambda,\mu\right)\right),
\]
with $d_\alpha\left(\lambda,\mu\right)$ being defined by
\[
d_\alpha\left(\lambda,\mu \right):=s_0^{2(1-\alpha)}|\theta_\lambda - \theta_{\mu}|^2 + s_0^{2\alpha}|x_\lambda - x_{\mu}|^2 +
\frac{s_0^2}{1+s_0^{2(1-\alpha)}|\theta_\lambda-\theta_\mu|^2}|\langle e_\lambda , x_\lambda - x_{\mu}\rangle|^2,
\]
where $s_0 = \mbox{min}(s_\lambda,s_{\mu})$ and $e_\lambda = \left(\cos(\theta_\lambda),-\sin(\theta_\lambda)\right)^T=R_{-\theta_\lambda}e_1$ is the co-direction.
\end{definition}

We emphasize that $\omega_\alpha$ certainly depends on the parametrizations $(\Lambda,\Phi_\Lambda)$ and $(\Delta,\Phi_\Delta)$.
However, in order not to overload the notation, we did not explicitly specify those, since it should always be clear from the
context.

We now come to one of the main results of this paper, which essentially states that two systems of $\alpha$-molecules are
almost orthogonal with respect to the $\alpha$-scaled index distance in the sense of a strong off-diagonal decay of the
associated cross-Gramian matrix. Due to this result a higher $\alpha$-scaled index distance can be interpreted as a lower
cross-correlation of associated $\alpha$-molecules. It should be noted that we only compare $\alpha$-molecules with the
same $\alpha$, since we aim to, for instance, transfer sparse approximation properties among those classes. It might though
be very interesting for future research to also let $\alpha$-molecules for different $\alpha$'s interact.

Let us now state the anticipated theorem on the cross-Gramian of two systems of $\alpha$-molecules. Its proof is technically
very involved and lengthy, wherefore we postpone it to Section \ref{subsec:proofalmostorth}.

\begin{theorem}\label{thm:almostorth}
Let $\alpha\in[0,1]$, and let $(m_\lambda)_{\lambda\in \Lambda}$ and $(p_{\mu})_{\mu\in\Delta}$ be two systems of
$\alpha$-molecules of order $(L,M,N_1,N_2)$. Further assume that there exists some constant $c > 0$ such that
\[
s_\lambda \ge c \quad\mbox{and}\quad s_\mu\ge c \quad \mbox{for all } \lambda\in\Lambda, \mu\in\Delta \mbox{ with } \left(s_\lambda , \theta_\lambda , x_\lambda\right)
:= \Phi_\Lambda(\lambda), \left(s_\mu , \theta_\mu , x_\mu\right) := \Phi_\Delta(\mu),
\]
and that there exists some constant $N\in\N$ such that
\[
L\geq 2N, \quad  M > 3N-\frac{3-\alpha}{2},\quad N_1 \geq N+\frac{1+\alpha}{2}, \quad \mbox{and} \quad N_2\geq 2N.
\]
Then
\[
\left|\left\langle m_\lambda ,p_{\mu}\right\rangle\right|\lesssim \omega_\alpha(\lambda,\mu)^{-N}
\quad \mbox{for all } \lambda\in\Lambda, \mu\in\Delta.
\]
\end{theorem}

This result provides us with a fundamental property of $\alpha$-molecules, which can be explored in various ways.
Perhaps one of the most notable applications is the classification and analysis of $\alpha$-molecules with
respect to their (sparse) approximation properties, which we will present in the following section.

\section{Sparse Approximations}
\label{sec:approx}

One main goal of introducing the framework of $\alpha$-molecules was to unify the treatment and analysis of sparse
approximation properties of multiscale systems constructed by applied harmonic analysis methodologies. In this
section we will now show that
\begin{itemize}
\item[(I)] $\alpha$-molecules can be categorized by their approximation behavior,
\item[(II)] sparse approximation results can be transferred from one system of $\alpha$-molecules to another,
\item[(III)] sparse approximation results can be concluded from the order of a system of $\alpha$-molecules.
\end{itemize}
Goal (I) will be discussed in Subsection \ref{subsec:sparsityequiv} and resolved by utilizing the notion of
sparsity equivalence from \cite{Grohs2011} and the novel notion of consistency of parameterizations. Goal (II)
will be analyzed in Subsection \ref{subsec:transfer} and we focus in particular on transferring sparse
approximation results from $\alpha$-curvelet and shearlet molecules. But the developed mechanisms can also
be employed for other systems. Finally, Goal (III) is studied in the case of optimally sparse approximation
of specific so-called cartoon-like functions in Subsection \ref{subsec:cartoon}. The basic idea in this part
will be that certain sparse approximation results known for $\alpha$-curvelets can be transferred to
sparsity equivalent systems leading, roughly speaking, to stand-alone conditions for the order of a
system of $\alpha$-molecules. Again, this approach can also be applied to other scenarios and shall also
show the power of this new framework.

Prior to this endeavour, we briefly recall some necessary aspects of approximation theory focussing on the
Hilbert space $L^2(\R^2)$, in which also $\alpha$-molecules are defined. Given some system $(m_\lambda)_{\lambda\in \Lambda} \subseteq L^2(\R^2)$,
one typically aims to efficiently represent/encode functions $f \in  L^2(\R^2)$ by the coefficients $c_\lambda\in\R$ of the expansion
\begin{align}\label{eq:expansion}
f = \sum_{\lambda\in \Lambda} c_\lambda m_\lambda.
\end{align}
Often efficiency can be improved by allowing the system $(m_\lambda)_{\lambda\in \Lambda}$ to be redundant. This
leads to the notion of a frame, which adds stability to redundancy (see, e.g., \cite{Christensen2003a}). A system of
functions $(m_\lambda)_{\lambda\in\Lambda}$ in $L^2(\R^2)$ forms a \emph{frame}, if there exist constants $A,B>0$,
called the \emph{frame bounds}, such that
\[
A\|f\|^2 \le \sum_{\lambda\in\Lambda} |\langle f,m_\lambda \rangle|^2 \le B\|f\|^2 \text{ for all }f\in L^2(\R^2).
\]
If $A$ and $B$ can be chosen equal the frame is called \emph{tight}. In case $A=B=1$ one speaks of a \emph{Parseval frame}.
The associated frame operator $S:L^2(\R^2)\rightarrow L^2(\R^2)$ is given by $Sf=\sum_{\lambda\in\Lambda} \langle f,m_\lambda \rangle m_\lambda$.
Since $S$ is always invertible, the system $(S^{-1}m_\lambda)_\lambda$ is also a frame, referred to as the \emph{canonical dual frame}.
It can be used to compute a particular sequence of coefficients in the expansion \eqref{eq:expansion} via
\[
c_\lambda=\langle f,S^{-1}m_\lambda \rangle, \quad \lambda\in\Lambda.
\]
This sequence however is usually not the only one possible. Unlike the expansion in a basis, a representation with
respect to a frame needs certainly not be unique. The canonical dual frame can also be used to express $f$ in terms
of the {\em frame coefficients} $(\langle f,m_\lambda \rangle)_\lambda$ by
\[
f = \sum_{\lambda\in \Lambda} \langle f,m_\lambda \rangle S^{-1} m_\lambda.
\]
In general, any system $(\tilde{m}_\lambda)_{\lambda\in \Lambda}$ satisfying this reconstruction formula for all
$f \in L^2(\R^2)$ is called an associated {\em dual frame}.

Let us now turn to the question of efficient encoding. In practice we have to restrict to finite expansions \eqref{eq:expansion},
which usually leads to an approximation error. Given a positive integer $N$, the best $N$-term approximation $f_N$ of some
function \pg{$f\in L^2(\R^2)$} with respect to the system $(m_\lambda)_\lambda$ is defined by
\[
    f_N = \argmin \Big\|f -  \sum_{\lambda \in \Lambda_N}c_\lambda m_\lambda\Big\|_2^2 \quad\mbox{s.t.}\quad \#\Lambda_N \le N.
\]
One can now analyze the rate at which the approximation error $\|f-f_N\|_2$ decays as $N \to \infty$. If we restrict the set
of data $f$ to a class $\cC \subseteq L^2(\R^2)$, we can say that a system $(m_\lambda)_\lambda$ provides {\em optimally
sparse approximations}, if this decay is the fastest among all systems in $L^2(\R^2)$ for each member $f$ of $\cC$.

The computation of the best $N$-term approximation by frames is not yet fully understood, even in the special case
of Parseval frames. Therefore, it is common to consider as a handier substitute the $N$-term approximation, obtained by
keeping the $N$ largest coefficients. Obviously, this approximation provides a bound for the best $N$-term approximation error.

\subsection{Categorization by Sparsity Equivalence}\label{subsec:sparsityequiv}

In this subsection, we aim to \ms{categorize $\alpha$-molecules} \pg{with respect to} their approximation
behavior. This will be achieved by the notion of sparsity equivalence from \cite{Grohs2011}
and the novel notion of ($\alpha,k$)-consistency, which will provide sufficient conditions
for two systems of $\alpha$-molecules \ms{to be sparsity equivalent.}

To build up intuition, we start by noticing that the $N$-term approximation rate achieved by a frame, is closely related
to the decay of the corresponding frame coefficients. Usually, the decay of a sequence -- \pg{sometimes} also called its sparsity
-- is measured by a strong or weak $\ell^p$-(quasi)-norm, for small $p>0$. Recall that the weak $\ell^p$-(quasi-)norm is
defined by
\[
\|(c_\lambda)_\lambda\|_{\omega\ell^p}:= \Big( \sup_{\varepsilon>0} \varepsilon^p \cdot \#\{\lambda: |c_\lambda|>\varepsilon \}\Big)^{1/p}.
\]
Every non-increasing rearrangement $(c^*_n)_{n\in\N}$ of $(c_\lambda)_\lambda\in\omega\ell^p$ satisfies
$\sup_{n>0} n^{1/p}|c^\ast_n| \pg{\asymp \|(c_\lambda)_\lambda\|_{\omega\ell^p}\lesssim \|(c_\lambda)_\lambda\|_{\ell^p}}$. One result showing that
membership of the coefficient sequence of $f$ in an $\ell^p$ space for small $p$ implies good $N$-term approximation
rates whenever the given representation system constitutes a frame is as follows. The respective proof can be found in
\cite{Kutyniok2010,Devore1998}), but for the convenience of the reader we also included it in Subsection \ref{subsec:proofdecayapprox}.

\begin{lemma}\label{lem:decayapprox}
Let $(m_\lambda)_{\lambda\in \Lambda}$ be a frame in $L^2(\R^2)$ and $f=\sum c_\lambda m_\lambda$ an expansion of
$f\in L^2(\R^2)$ with respect to this frame. If $(c_\lambda)_\lambda\in \omega\ell^{2/(p+1)}(\Lambda)$ for some $p>0$,
then the $N$-term approximation rate for $f$ achieved by keeping the $N$ largest coefficients is at least of order $N^{-p/2}$, i.e.
\[
 \| f-f_N \|_2^2 \lesssim N^{-p}.
\]
\end{lemma}

As illustrated by Lemma~\ref{lem:decayapprox} the decay rate of the frame coefficients determines the $N$-term approximation rate.
\pg{In particular, if the sequence $\left(\left\langle f, m_\lambda\right\rangle\right)_{\lambda\in \Lambda}$
of frame coefficients lies in an $\ell^p$ space
for $p<1$, then
the best approximation rate of the \emph{dual frame}
$\left(\tilde{m}_\lambda\right)_{\lambda\in\Lambda}$
is at least of order $N^{-(1/p -1/2)}$. In terms of
signal compression this is exactly what one hopes for:
from simply keeping the $N$ largest frame coefficients (which can be encoded by order $N$ bits)
we can reconstruct the original signal $f$
up to a precision of order $N^{-(1/p -1/2)}$.
}
Let us assume that we have two systems $(m_\lambda)_{\lambda\in \Lambda}$ and $(p_\mu)_{\mu \in \Delta}$ in $L^2(\R^2)$ and expansion
coefficients for $f\in L^2(\R^2)$ with respect to these two systems. Then these systems provide the same $N$-term approximation rate
for $f$ if the corresponding expansion coefficients have similar decay, e.g.\ if they belong to the same $\ell^p$-space.



\begin{prop}
\label{prop:justification}
\pg{Let $0 < p \le 1$, let $f \in L^2(\R^2)$, and let $(m_\lambda)_{\lambda\in \Lambda}$ and $(p_\mu)_{\mu \in \Delta}$ be \ms{frames} such that
\[
\left\|\left(\langle  m_{\lambda},p_\mu\rangle\right)_{\lambda\in \Lambda, \mu\in \Delta}\right\|_{\ell^p\to \ell^p}<\infty.
\]
Moreover, let $(\tilde{m}_\lambda)_{\lambda\in \Lambda}$ be a dual frame for $(m_\lambda)_{\lambda\in \Lambda}$.
\ms{Then $(\langle f,\tilde{m}_\lambda \rangle)_\lambda \in \ell^p(\Lambda)$
implies $(\langle f,p_\mu \rangle)_\mu \in \ell^p(\Delta)$.}
In particular, $f$ can be encoded by the $N$ largest
frame coefficients from $(\langle f,p_\mu \rangle)_\mu$ up to accuracy $\lesssim N^{-(1/p-1/2)}$.}
\end{prop}

\begin{proof}
For fixed $\mu \in \Delta$, we have
\[
\langle f,p_\mu \rangle = \Big\langle \sum_{\lambda \in \Lambda} \ip{f}{\tilde{m}_\lambda} m_\lambda,p_\mu \Big\rangle =
\sum_{\lambda \in \Lambda} \ip{f}{\tilde{m}_\lambda} \left\langle  m_\lambda,p_\mu \right\rangle.
\]
Thus $(\langle f,\tilde{m}_\lambda \rangle)_\lambda \in \ell^p(\Lambda)$ and
$\|\left(\langle  m_{\lambda},p_\mu\rangle\right)_{\lambda\in \Lambda, \mu\in \Delta}\|_{\ell^p\to \ell^p}<\infty$
imply $(\langle f,p_\mu \rangle)_\mu \in \ell^p(\Delta)$.
\end{proof}

This result motivates the following notion of sparsity equivalence initially introduced in~\cite{Grohs2011} for
parabolic molecules.
\pg{

\begin{definition}
Let  $0 < p \le 1$, and let $(m_\lambda)_{\lambda\in \Lambda}$ and $(p_\mu)_{\mu \in \Delta}$ be
frames. Then $(m_\lambda)_{\lambda\in \Lambda}$ and $(p_\mu)_{\mu \in \Delta}$
are {\em sparsity equivalent in $\ell^p$}, if
\[
\left\|\left(\langle {m}_{\lambda},p_\mu\rangle\right)_{\lambda\in \Lambda, \mu\in \Delta}\right\|_{\ell^p\to \ell^p}<\infty.
\]
\end{definition}
}
\ms{
The concept of sparsity equivalence allows to extend approximation properties from one anchor system to other systems,
if the coefficient decay of the anchor system is known.
This notion however, does not provide an equivalence relation.
}
 We further emphasize that \ms{sparsity equivalence} depends sensitively on the regularity parameter $0 < p \le 1$.



Having introduced sparsity equivalence for \pg{frames}, we now require sufficient conditions
for two systems of $\alpha$-molecules to be sparsity equivalent. We expect this to depend on the one hand on
the respective orders of those systems. On the other hand, now the relation of the parametrizations becomes crucial leading
to the notion of ($\alpha,k$)-consistency. To motivate this novel concept, we first recall a simple estimate for the
operator norm of a matrix on discrete $\ell^p$ spaces from \cite{Grohs2011}.

\begin{lemma}[\cite{Grohs2011}]\label{lem:matnorm}
Let $\Lambda,\Delta$ be two discrete index sets, and let ${\bf A} : \ell^p(\Lambda)\to \ell^p(\Delta)$, $p>0$ be a linear mapping
defined by its matrix representation ${\bf A} = \left(A_{\lambda,\mu}\right)_{\lambda\in \Lambda ,\, \mu\in \Delta}$.
Then we have the bound
\[
\|{\bf A}\|_{\ell^p(\Lambda)\to \ell^p(\Delta)}\le\max\left\{\sup_{\lambda \in \Lambda}\sum_{\mu \in \Delta}|A_{\lambda,\mu}|^{q},
\sup_{\mu \in \Delta}\sum_{\lambda \in \Lambda}|A_{\lambda,\mu}|^q\right\}^{1/q},
\]
where $q:=\min\{1,p\}$.
\end{lemma}

Aiming for sufficient conditions for the right hand side -- in the situation \msch{of ${\bf A}$} being the Gramian of two
systems of $\alpha$-molecules -- to be finite, also taking the estimate provided in Theorem \ref{thm:almostorth}
into account, it seems appropriate to introduce the following notion.

\begin{definition}\label{def:kadmiss}
Let $\alpha \in [0,1]$ and $k > 0$. Two parametrizations $(\Lambda,\Phi_\Lambda)$ and $(\Delta,\Phi_\Delta)$
are called \emph{($\alpha,k$)-consistent}, if
\[
\sup_{\lambda\in \Lambda}\sum_{\mu\in\Delta}\omega_{\alpha}\left(\lambda,\mu\right)^{-k} <\infty \quad\text{and}\quad \sup_{\mu\in \Delta}\sum_{\lambda\in\Lambda}
\omega_{\alpha}\left(\lambda,\mu\right)^{-k} <\infty.
\]
\end{definition}

As expected, this notion leads to a convenient sufficient condition for sparsity equivalence of $\alpha$-molecules.

\begin{theorem}
\label{theo:sparseeqivmol}
Let $\alpha \in [0,1]$, $k > 0$, and $0<p\le1$.  Let $(m_\lambda)_{\lambda\in\Lambda}$ and $(p_\mu)_{\mu\in\Delta}$ be two frames of $\alpha$-molecules of order
$(L,M,N_1,N_2)$ with ($\alpha,k$)-consistent parametrizations $(\Lambda,\Phi_\Lambda)$ and $(\Delta,\Phi_\Delta)$ satisfying
\[
L\geq 2\frac{k}{p},\quad  M> 3\frac{k}{p} - \frac{3-\alpha}{2} ,\quad N_1 \geq \frac{k}{p}+\frac{1+\alpha}{2} ,  \quad \mbox{and} \quad  N_2\geq 2\frac{k}{p}.
\]
Then $(m_\lambda)_{\lambda\in\Lambda}$ and $(p_\mu)_{\mu\in\Delta}$ are sparsity equivalent in $\ell^p$.
\end{theorem}

\begin{proof}
By Lemma \ref{lem:matnorm}, it suffices to prove that
\[
 \max \left\{ \sup_{\lambda\in\Lambda}\sum_{\mu\in \Delta} |\langle m_{\lambda},p_\mu\rangle|^p, \sup_{\mu\in\Delta}\sum_{\lambda\in \Lambda}
|\langle m_{\lambda},p_\mu\rangle|^p\right\}^{1/p} <\infty.
\]
Since, by Theorem \ref{thm:almostorth}, we have
\[
|\langle m_{\lambda},p_\mu\rangle|\lesssim \omega_{\alpha}(\lambda,\mu)^{-\frac{k}{p}},
\]
we can conclude that
\[
\max \left\{ \sup_{\lambda\in\Lambda}\sum_{\mu\in \Delta}|\langle m_{\lambda},p_\mu\rangle|^p,\sup_{\mu\in\Delta}\sum_{\lambda\in \Lambda}|\langle m_{\lambda},p_\mu\rangle|^p\right\}%
\lesssim\max \left\{ \sup_{\lambda\in\Lambda}\sum_{\mu\in \Delta}\omega_{\alpha}(\lambda,\mu)^{-k},\sup_{\mu\in\Delta}\sum_{\lambda\in \Lambda}\omega_{\alpha}(\lambda,\mu)^{-k}\right\}
\]
with the expression on the right hand side being finite due to the $(\alpha,k)$-consistency of the parametrizations $(\Lambda,\Phi_\Lambda)$ and $(\Delta,\Phi_\Delta)$.
The proof is completed.
\end{proof}

Thus, as long as the parametrizations are consistent, the sparsity equivalence can be controlled by the order of the
molecules. Recall that higher order means better time-frequency localization and higher moments. Hence, intuitively,
the smaller $p$ is (i.e., the more sparsity is promoted) and the less consistent the two frames of $\alpha$-molecules
are, the better their time-frequency localization and the higher their moments need to be in order for them \ms{to be
sparsity equivalent.}

\subsection{Transfer of Sparse Approximation Results}\label{subsec:transfer}

We next aim to investigate situations in which we can actually transfer sparse approximation results based on
Theorem \ref{theo:sparseeqivmol}. In Section \ref{sec:examples}, we provided a range of prominent multiscale
systems which are encompassed by the framework of $\alpha$-molecules. It became apparent that most of such
can be regarded as instances of either $\alpha$-curvelet or $\alpha$-shearlet molecules. Thus, it seems
natural to first analyze those systems with respect to ($\alpha,k$)-consistency.

For this, we recall that $\alpha$-curvelet and $\alpha$-shearlet molecules are associated with $\alpha$-curvelet
and $\alpha$-shearlet parametrizations (cf. Definitions \ref{defi:alphacurveletpara} and \ref{defi:alphashearletpara}).
The following result shows that indeed those parametrizations satisfy the consistency requirement for
any $k > 2$.

\begin{theorem}\label{thm:consist}
Let $\alpha\in[0,1]$ and $(\Lambda,\Phi_\Lambda)$ and $(\Delta,\Phi_\Delta)$ be either $\alpha$-curvelet or $\alpha$-shearlet parametrizations.
Then $(\Lambda,\Phi_\Lambda)$ and $(\Delta,\Phi_\Delta)$ are ($\alpha,k$)-consistent for all $k>2$.
\end{theorem}

The proof of this result relies on the following technical lemma, whose proof we outsource to Subsection \ref{subsec:proofauxest}.

\begin{lemma}\label{lem:auxest}
Let $\alpha\in[0,1]$, let $N>2$, and let $\mu=(s_\mu,\theta_\mu,x_\mu)\in\mathbb{P}$ be an arbitrary fixed point of the parameter space $\mathbb{P}$.
\begin{enumerate}
\item[(i)] For $(\Lambda^c,\Phi^c)$ being an $\alpha$-curvelet parametrization, there exists a constant $C>0$ independent of $\mu$ and $s_\lambda$ such that
\[
\sum_{\substack{\lambda\in\Lambda^c\\ s_\lambda \text{ fixed}}} (1+d_\alpha(\lambda,\mu))^{-N} \le C\cdot \max\Big\{\frac{s_\lambda}{s_\mu},1\Big\}^2.
\]
\item[(ii)] For $(\Lambda^s,\Phi^s)$ being an $\alpha$-shearlet parametrization, there exists a constant $C>0$ independent of $\mu$ and $s_\lambda$ such that
\[
\sum_{\substack{\lambda\in\Lambda^s \\ s_\lambda\text{ fixed}}} (1+d_\alpha(\lambda,\mu))^{-N} \le C\cdot \max\Big\{\frac{s_\lambda}{s_\mu},1\Big\}^2.
\]
\end{enumerate}
\end{lemma}

Since the main technical difficulties are contained in the proof of this lemma, the actual proof of Theorem \ref{thm:consist}
now just takes a few lines.

\begin{proof}[Proof of Theorem \ref{thm:consist}]
We aim to prove that
\[
\sup_{\mu\in \Delta}\sum_{\lambda\in\Lambda} \omega_{\alpha}\left(\mu,\lambda\right)^{-k} <\infty.
\]
By the definition of $\omega_{\alpha}\left(\mu,\lambda\right)$, for every $\mu\in\Delta$, we need to consider
\begin{equation}\label{eq:prove2}
\sum_{j\in\N_0}\sum_{\substack{\lambda \in \Lambda \\ s_\lambda  = \sigma^j}}\max\Big\{ \frac{s_\lambda}{s_\mu},\frac{s_\mu}{s_\lambda} \Big\}^{-k}\left(1 + d_\alpha(\mu,\lambda)\right)^{-k}.
\end{equation}
According to Lemma \ref{lem:auxest}, for each fixed $j\in\N_0$ and $k>2$,
\[
\sum_{\lambda \in \Lambda, s_\lambda  = \sigma^j}(1 + d_\alpha(\mu,\lambda))^{-k}\lesssim\max\Big\{\frac{s_\lambda}{s_\mu},1\Big\}^2.
\]
Let now $j^\prime\in\N_0$ be such that $s_{\mu} = \sigma^{j'}$. Then \eqref{eq:prove2} can be estimated by
\[
\sum_{j\in\N_0}  \max\Big\{\frac{s_\lambda}{s_\mu},1\Big\}^2 \max\Big\{\frac{s_\lambda}{s_\mu},\frac{s_\mu}{s_\lambda}\Big\}^{-k}
\hspace*{-0.2cm}\le \sum_{j\in\N_0}  \max\Big\{\frac{s_\lambda}{s_\mu},\frac{s_\mu}{s_\lambda}\Big\}^{2-k}
\hspace*{-0.2cm}= \sum_{j\in\N_0} \sigma^{|j-j'|(2-k)} \le 2\sum_{j\in\N_0} \sigma^{j(2-k)} = C <\infty,
\]
where $C$ is independent of $j^\prime$, and thus of $\mu$. This finishes the proof.
\end{proof}

This now allows us to actually derive novel results by a simple transfer using Theorem \ref{theo:sparseeqivmol}
\pg{and Proposition \ref{prop:justification}}. \ms{In fact, we will demonstrate how to derive the much more general Theorems~\ref{theo:mainsparsity2} and \ref{theo:mainsparsity}
from one particular result, namely Theorem~\ref{thm:acurveapprox}, by using the machinery developed here.}
As we shall see below in Subsection \ref{subsec:cartoon}, this will lead to a number
of novel results concerning best $N$-term approximations for cartoon-like images, also defined at this point.

\subsection{Sparse Approximation of Cartoon-like Functions}\label{subsec:cartoon}

We finally show how the framework of $\alpha$-molecules allows to prove approximation results in a more
systematic way. It provides, for instance, an explanation for similar approximation rates observed for
different systems. From the viewpoint of $\alpha$-molecules this is a natural consequence of the time-frequency
localization of the systems.

\pagebreak

The general strategy is as follows. If an approximation result of a specific system of $\alpha$-molecules is known
-- in the sequel $\alpha$-curvelets and the class of cartoon-like functions are considered --,  and it can be shown
that a class of $\alpha$-molecules with certain conditions on the control parameters (the parametrization and the
order) satisfies the hypotheses of Theorem \ref{theo:sparseeqivmol},  i.e., they are all sparsity equivalent to this
specific system, they automatically inherit its known approximation behavior.

To present one application of this general concept, we start by introducing the model situation we will consider,
followed by recalling the known sparse approximation result we aim to transfer. Finally, we will obtain novel
stand-alone sparse \msch{approximation results} for a class of $\alpha$-molecules with sufficiently large order and
certain consistency conditions on their parametrization.

\subsubsection{Model Situation}

The general continuum model for image data is the space $L^2(\R^2)$. However, for real-life images like photos for
example, such a general model is usually not needed and seems to be a too broad approach. Based on the observation
that natural images typically consist of piecewise smooth patches -- and taking into account that the neurons in
the visual cortex are highly directional sensitive, thereby making anisotropic features always predominant -- it can be
further refined, giving rise to the class of so-called \emph{cartoon-like functions}.

The first such model $\cE^{2}(\RR^2)$ was introduced in \cite{Don01}. It postulates that natural images consist of
$C^2$-regions separated by piecewise smooth $C^2$-curves. Since then several extensions of the original model have
been made and studied, starting with the work in \cite{Kutyniok2012}. By now cartoon-like functions have been
established as a widely used standard model, in particular for natural images.

In the sequel, we consider an extension of the original model, first considered in \cite{Kutyniok2012}, which are
images consisting of two smooth $C^\beta$-regions, $\beta\in(1,2]$, separated by a piecewise smooth $C^\beta$-curve.
The formal definition is as follows.

\begin{definition}
For $\beta\in(1,2]$, the model class $\cE^{\beta}(\RR^2)$ of {\em cartoon-like functions} is given by
\[
\cE^{\beta}(\RR^2) = \{f \in L^2(\RR^2) : f = f_0 + f_1 \cdot \chi_{B}\},
\]
where $f_0,f_1 \in C_0^{\beta}([0,1]^2)$ and $B \subset [0,1]^2$ is a Jordan domain with a regular closed piecewise
smooth $C^{\beta}$-curve as boundary.
\end{definition}

Beginning with \cite{Don01} it was established in a series of papers \cite{Kutyniok2012,Kei13,GKKScurve2014}, that the
optimally achievable decay rate of the $N$-term approximation error for $f \in \cE^{\beta}(\RR^2)$ with $\beta\in(1,2]$,
in any dictionary under the natural assumption of polynomial depth search, is
\[
\norm{f - f_N}_2^2 \asymp N^{-\beta}, \quad \mbox{as } N \to \infty.
\]
Furthermore, it was proven in \cite{CD04,GKKScurve2014,GL07,Kei13,Kutyniok2010} that $\alpha$-curvelet and
$\alpha$-shearlet systems attain this rate up to a log-factor, provided that $\alpha=\beta^{-1}$. Thus,
these systems behave similarly concerning their sparse approximation properties, and the framework of
$\alpha$-molecules will not only provide us with an explanation, but also enable us to derive similar results
for a much wider class of multiscale systems.

\subsubsection{Sparse Approximation with $\alpha$-Curvelets}

Next, we require a concrete system of $\alpha$-molecules, which establishes the optimal $N$-term approximation rate with
respect to the class $\mathcal{E}^\beta(\R^2)$.

A suitable choice for the reference system is the tight frame of $\alpha$-curvelets $C_\alpha(W^{(0)},W,V)$ given by
Definition \ref{defi:curvelets}. By Proposition~\ref{prop:acurvmol}, it constitutes a system of $\alpha$-molecules of
order $(\infty,\infty,\infty,\infty)$. Moreover, it was shown in \cite{GKKScurve2014} that it provides (up to a $\log$-factor)
optimal $N$-term approximation for the class of cartoon-like functions $\mathcal{E}^\beta(\R^2)$ for $\beta=\alpha^{-1}$.

\begin{theorem}[\cite{GKKScurve2014}]\label{thm:acurveapprox}
Let $\alpha\in[\frac{1}{2},1)$ and $\beta=\alpha^{-1}$. The tight frame of $\alpha$-curvelets $C_\alpha(W^{(0)},W,V)$
provides almost optimal sparse approximations for cartoon-like functions in $\cE^{\beta}(\RR^2)$. More precisely,
there exists some constant $C>0$ such that for every $f\in\mathcal{E}^\beta(\R^2)$
\[
\|f-f_N\|_2^2\le CN^{-\beta} \cdot \left(log_2 N\right)^{\beta+1} \quad \mbox{as }  N \rightarrow \infty.
\]
where $f_N$ denotes the $N$-term approximation of $f$ obtained by choosing the $N$ largest coefficients.
\end{theorem}

More precisely, it was proved in \cite{GKKScurve2014} that the curvelet coefficients belong to $\omega\ell^{p}(\Lambda^c)$
for every $p>\frac{2}{1+\beta}$, $\Lambda^c$ being the \msch{curvelet index set}.

We mention that this type of optimal sparse approximation focuses on the cases of $\alpha\in[\frac{1}{2},1)$. Certainly,
once approximation results are established for a reference system for some $\alpha \in [0,\frac12)$, the general machinery can
be applied as well.

\subsubsection{Optimality Result}

Via Theorem~\ref{theo:sparseeqivmol} and the notion of sparsity equivalence, it is now possible to transfer the approximation
rate established in Theorem~\ref{thm:acurveapprox} to more general systems of $\alpha$-molecules. For this, let $(\Lambda^c,\Phi^c)$
denote the parametrization of the tight frame of $\alpha$-curvelets $C_\alpha(W^{(0)},W,V)$.

Finally, we can formulate and prove our main result concerning the approximation properties of $\alpha$-molecules, which
identifies a large class of multiscale systems with (almost) optimal approximation performance for the class of cartoon-like
functions $\cE^{\beta}(\RR^2)$. By Theorem~\ref{thm:consist}, the required condition (i) holds in particular for the
curvelet and shearlet parametrizations, for $k>2$. Thus, this result allows a simple and systematic derivation not only of the
results in \cite{CD04,GKKScurve2014,GL07,Kei13,Kutyniok2010}, but for a much larger class of $\alpha$-molecules.

\begin{theorem}\label{theo:mainsparsity2}
Let $\alpha\in[\frac{1}{2},1)$ and $\beta=\alpha^{-1}$. Assume that, for some $k>0$, a tight frame $(m_\lambda)_{\lambda\in \Lambda}$ of
$\alpha$-molecules satisfies the following two conditions:
\begin{itemize}
\item[(i)] its parametrization $(\Lambda,\Phi_\Lambda)$ and $(\Lambda^c,\Phi^c)$ are ($\alpha,k$)-consistent,
\item[(ii)] its order $(L,M,N_1,N_2)$ satisfies
\[
L\geq k(1+\beta) ,\quad  M \geq \frac{3k}{2} (1+\beta) + \frac{\alpha-3}{2} , \quad N_1 \geq \frac{k}{2} (1+\beta) +\frac{1+\alpha}{2} , \quad \mbox{and} \quad N_2\geq k(1+\beta).
\]
\end{itemize}
Then $(m_\lambda)_{\lambda\in \Lambda}$ possesses an almost optimal $N$-term approximation rate for the class
of cartoon-like functions $\cE^{\beta}(\RR^2)$, i.e., for all $f\in\mathcal{E}^\beta(\R^2)$,
\[
\|f-f_N\|_2^2 \lesssim N^{-\beta+\varepsilon}, \quad\varepsilon >0 \text{ arbitrary},
\]
where $f_N$ denotes the $N$-term approximation obtained from the $N$ largest frame coefficients.
\end{theorem}

This result can also be extended to general frames, which then provides this approximation behavior for any associated
dual frame. Certainly, it suffices to prove
only this theorem, which includes Theorem \ref{theo:mainsparsity2} as a special case.

\begin{theorem}\label{theo:mainsparsity}
Let $\alpha\in[\frac{1}{2},1)$ and $\beta=\alpha^{-1}$. Assume that, for some $k>0$, a frame $(m_\lambda)_{\lambda\in \Lambda}$ of
$\alpha$-molecules satisfies the following two conditions:
\begin{itemize}
\item[(i)] its parametrization $(\Lambda,\Phi_\Lambda)$ and $(\Lambda^c,\Phi^c)$ are ($\alpha,k$)-consistent,
\item[(ii)] its order $(L,M,N_1,N_2)$ satisfies
\[
L\geq k(1+\beta) ,\quad  M \geq \frac{3k}{2} (1+\beta) + \frac{\alpha-3}{2} , \quad N_1 \geq \frac{k}{2} (1+\beta) +\frac{1+\alpha}{2} , \quad \mbox{and} \quad N_2\geq k(1+\beta).
\]
\end{itemize}
Then each dual frame $(\tilde{m}_\lambda)_{\lambda\in \Lambda}$ possesses an almost optimal $N$-term approximation rate for the class
of cartoon-like functions $\cE^{\beta}(\RR^2)$, i.e., for all $f\in\mathcal{E}^\beta(\R^2)$,
\[
\|f-f_N\|_2^2 \lesssim N^{-\beta+\varepsilon}, \quad\varepsilon >0 \text{ arbitrary},
\]
where $f_N$ denotes the $N$-term approximation obtained from the $N$ largest frame coefficients.
\end{theorem}

\begin{proof}
Let $C_\alpha(W^{(0)},W,V)=(\psi_\mu)_{\mu\in\Lambda^c}$ be the tight frame of $\alpha$-curvelets defined in Definition
\ref{defi:curvelets}, and let $f\in \cE^{\beta}(\RR^2)$. By \cite[Thm. 4.2]{GKKScurve2014}, the sequence of curvelet
coefficients $(\theta_\mu)_\mu$ given by $\theta_\mu=\langle f, \psi_\mu \rangle$ belongs to $\omega\ell^{p}(\Lambda^c)$
for every $p>\frac{2}{1+\beta}$. Since $\omega\ell^{p}\hookrightarrow \ell^{p+\epsilon}$ for arbitrary $\varepsilon>0$,
this further implies $(\theta_\mu)_\mu \in \ell^{p}(\Lambda^c)$ for every $p>\frac{2}{1+\beta}$.

Let now
\[
f=\sum_{\lambda\in\Lambda} c_\lambda \tilde{m}_\lambda
\]
be the canonical expansion of $f$ with respect to the dual frame $(\tilde{m}_\lambda)_\lambda$, with frame coefficients
$(c_\lambda)_\lambda$ given by
\[
c_\lambda= \langle f, m_\lambda\rangle = \sum_{\mu} \langle  \psi_\mu, m_\lambda\rangle \theta_\mu.
\]
Thus, they are related to the curvelet coefficients $(\theta_\mu)_\mu$ by the cross-Gramian $(\langle \psi_\mu,m_\lambda \rangle)_{\mu,\lambda}$.
By Theorem~\ref{theo:sparseeqivmol}, conditions (i) and (ii) guarantee that the frame $(m_\lambda)_{\lambda\in\Lambda}$ is sparsity equivalent
to $(\psi_\mu)_{\mu\in\Lambda^c}$ in $\ell^p$ for every $p>\frac{2}{1+\beta}$. This implies that the cross-Gramian is a bounded operator
$\ell^p(\Lambda^c)\rightarrow\ell^p(\Lambda)$, which maps $(\theta_\mu)_\mu$ to $(c_\lambda)_\lambda$. Hence, $(c_\lambda)_\lambda\in\ell^{p}(\Lambda)$
for every $p>\frac{2}{1+\beta}$. The embedding $\ell^{p}\hookrightarrow \omega\ell^{p}$ then proves $(c_\lambda)_\lambda\in\omega\ell^{p}(\Lambda)$
for every $p>\frac{2}{1+\beta}$. Finally, for arbitrary $\varepsilon>0$, the application of Lemma~\ref{lem:decayapprox} yields
\[
\|f-f_N\|_2^2 \lesssim N^{-\beta+\varepsilon},
\]
where $f_N$ denotes the $N$-term approximation with respect to the system $(\tilde{m}_\lambda)_\lambda$ obtained by
choosing the $N$ largest coefficients.
\end{proof}
\pg{
Taking into account Proposition~\ref{prop:ashearmol} and Theorem~\ref{thm:consist}, the statement of Theorem~\ref{theo:mainsparsity2},
for instance, implies the following novel result concerning cartoon approximation with band-limited $\beta$-shearlet systems.
\begin{theorem}\label{theo:bandlimshearapprox}
Let $\beta\in(1,2]$, and let $SH\big( \phi,\psi,\tilde{\psi};c,\beta \big)$ be a frame of cone-adapted $\beta$-shearlets
obtained from band-limited generators as in Definition~\ref{def:betashearsystem}.
Then each dual frame possesses an almost optimal $N$-term approximation rate for the class
of cartoon-like functions $\cE^{\beta}(\RR^2)$, i.e., for all $f\in\mathcal{E}^\beta(\R^2)$, we have
\[
\|f-f_N\|_2^2 \lesssim N^{-\beta+\varepsilon}, \quad\varepsilon >0 \text{ arbitrary},
\]
where $f_N$ denotes the $N$-term approximation obtained from the $N$ largest frame coefficients.
\end{theorem}
}
\section{Proofs}
\label{sec:proofs}

\subsection{Proofs of Subsection \ref{subsec:shearletmolecules}}

\subsubsection{Proof of Proposition \ref{prop:shearletmol}}
\label{subsec:proofshearletmol}

We confine the discussion to $\varepsilon = 0$, the other case being analogous, and suppress the superscript $\varepsilon$
in our notation. It is sufficient to show that, for each $\lambda=(\varepsilon,\ell,j,k)\in\Lambda^s$, the function
\[
g^{(\lambda)}(\cdot):=\psi_\lambda\left(A_{\alpha,\sigma^j} S_{\ell,j} R_{\theta_\lambda}^{-1} A^{-1}_{\alpha,s_\lambda}\cdot\right)
\]
satisfies (\ref{eq:molcond}).

For this, first note that the Fourier transform of $g^{(\lambda)}$ is given by
\[
\hat g^{(\lambda)}(\cdot) =\hat \psi_\lambda \left(A_{\alpha,\sigma^{-j}} S_{\ell,j}^{-T}R_{\theta_\lambda}^T A_{\alpha,s_\lambda}\cdot\right).
\]
Let us now examine the `transfer matrix' $T:=R_{\theta_\lambda}(S_{\ell,j})^{-1}$. Since $\theta_\lambda=\arctan(-\ell \eta_j)$,
we have
\[
S_{\ell,j}= \left(\begin{array}{cc}1 & \ell \eta_j\\0&1\end{array}\right)=\begin{pmatrix} 1 & -\tan(\theta_\lambda) \\ 0 & 1 \end{pmatrix}.
\]
Using $0=\tan(\theta_\lambda)\cos(\theta_\lambda) -\sin(\theta_\lambda)$, we obtain
\begin{align*}
T= \begin{pmatrix} \cos\theta_\lambda & 0 \\ \sin\theta_\lambda & \tan(\theta_\lambda)\sin(\theta_\lambda) +\cos(\theta_\lambda)  \end{pmatrix}
=\begin{pmatrix} \cos\theta_\lambda & 0 \\ \sin\theta_\lambda & \cos(\theta_\lambda)^{-1}  \end{pmatrix} =:  \begin{pmatrix} a & 0 \\ b & c \end{pmatrix},
\end{align*}
where the quantities $a,b,c$ depend on the index $\lambda\in\Lambda^s$.

Next we recall that $|\ell|\lesssim \sigma^{j(1-\alpha)}$ and $\eta_j\asymp \sigma^{-j(1-\alpha)}$, which yields $|\ell\eta_j|\lesssim1$. This implies
the existence of $0<\delta<\frac{\pi}{2}$ such that $|\theta_\lambda|=|\arctan(-\ell \eta_j)|\le \delta$ for all $\lambda\in\Lambda^s$.
As a consequence, we have
\begin{align}\label{eq:quantabc}
0<\cos(\delta)\le a \le1, &&  1\le c \le \cos(\delta)^{-1}<\infty, && |b|\le \sin(\delta).
\end{align}
Thus, the quantities $a,b,c$ are uniformly bounded in modulus. Furthermore, $a$ and $c$ are strictly positive and bounded uniformly from
below by $\cos(\delta)$.

Next, observe that the matrix $A_{\alpha,\sigma^{-j}} S_{\ell,j}^{-T} R_{\theta_\lambda}^T A_{\alpha,s_\lambda}=A_{\alpha,\sigma^{-j}} T^T A_{\alpha,\sigma^j}$ has the form
\[
\left(\begin{array}{cc}a & \sigma^{-j(1-\alpha)}b \\0&c\end{array}\right).
\]
Note that $|\sigma^{-j(1-\alpha)}|\le1$ for every $j\in\N_0$. Thus, using the uniform boundedness of $|a|,|b|,|c|$ and the chain rule,
we can estimate for any $|\rho |\le L$:
\begin{align}\label{eq:chainest}
|\partial^{\rho} \hat g^{(\lambda)}(\xi)| &\lesssim \sup_{|\gamma|\le L} \left|\partial^\gamma \hat \psi_\lambda\left(\left(\begin{array}{cc}a & \sigma^{-j(1-\alpha)}b \\
0&c\end{array}\right)\xi\right)\right|.
\end{align}
Then we utilize the moment estimate (\ref{eq:shearcond}) for $\hat\psi$. This gives us the moment property required in (\ref{eq:molcond}),
\begin{align*}
|\partial^{\rho} \hat g^{(\lambda)}(\xi)| \lesssim \left( \sigma^{-j}+|\xi_1|+ 2\cdot \sigma^{-j(1-\alpha)}|\xi_2|\right)^{M} \lesssim \left( s_\lambda^{-1}+|\xi_1|+s_\lambda^{-(1-\alpha)}|\xi_2|\right)^{M}.
\end{align*}

It remains to show the decay of $\partial^{\rho} \hat g^{(\lambda)}$ for large frequencies $\xi$. We obtain from \eqref{eq:chainest} and the decay estimate in (\ref{eq:shearcond}),
\begin{align*}
|\partial^{\rho} \hat g^{(\lambda)}(\xi)|\lesssim \left \langle\left|\left(\begin{array}{cc}a & s_\lambda^{-(1-\alpha)}b\\
0&c\end{array}\right)\xi \right|\right\rangle^{-N_1} \langle c\xi_2 \rangle^{-N_2}\lesssim\left\langle |\xi|\right\rangle^{-N_1}\langle \xi_2\rangle^{-N_2}.
\end{align*}
The last estimate is a consequence of \eqref{eq:quantabc}. To verify this we write
\[
\left(\begin{array}{cc}a & \sigma^{-j(1-\alpha)}b \\0&c\end{array}\right) =
\left(\begin{array}{cc}a & 0 \\0&c\end{array}\right) \left(\begin{array}{cc}1 & h \\0&1\end{array}\right)=:\diag(a,c)\cdot S_h
\]
with $h= \sigma^{-j(1-\alpha)}b/a$. Due to \eqref{eq:quantabc} the shear parameter $h$ is bounded in modulus, which implies $|S_{h}\xi| \asymp |\xi|$ for $\xi\in\R^2$.
Finally, $|\diag(a,c) \xi| \asymp |\xi|$ and $|c\xi_2|\asymp|\xi_2|$ also by \eqref{eq:quantabc}.


\subsubsection{Proof of Proposition \ref{prop:ashearmol}(iii)}
\label{subsec:proofashearmol}

It suffices to prove that $SH(\phi,\psi,\tilde{\psi};c,\beta)$ is a system of $\alpha$-shearlet molecules for $\alpha=\beta^{-1}$ of
order $(L,M-L,N_1,N_2)$, where $L\in\{0,\ldots,M\}$, with the parameters of the $\alpha$-shearlet parametrization being given by
$\tau=c$, $\sigma=2^{\beta/2}$, $\eta_j=\sigma^{-j(1-\alpha)}$ and $L_j=\lceil \sigma^{j(1-\alpha)} \rceil$.

First, we name and index the functions of the system $SH(\phi,\psi,\tilde{\psi};c,\beta)$ in the following way. For $j\ge0$,
$\ell\in\Z$ with $|\ell|\le \lceil 2^{j(\beta-1)/2} \rceil$ and $k\in\Z^2$ we let
\begin{align*}
\psi_{(0 , j, \ell , k)} &:=\psi_{j,\ell,ck}= 2^{j(\beta+1)/4}\psi(S_\ell A_{\beta^{-1},2^{j\beta/2}}\cdot-ck ), \\
\psi_{(1 , j, \ell , k)} &:= \tilde{\psi}_{j,\ell,ck}= 2^{j(\beta+1)/4}\tilde{\psi}(S^T_\ell \tilde{A}_{\beta^{-1},2^{j\beta/2}}\cdot-ck ).
\end{align*}
At the coarse scale we put $\psi_{(0 , -1, 0 , k)} := \phi_{ck} = \phi(\cdot - ck)$ for $k\in\Z^2$.

Since $\alpha=\beta^{-1}$ and $\sigma=2^{\beta/2}$ the scaling matrix $A_{\beta^{-1},2^{j\beta/2}}$ can be rewritten in the form
\begin{align*}
A_{\beta^{-1},2^{j\beta/2}}=\begin{pmatrix} 2^{j\beta/2} & 0 \\ 0 & 2^{j/2}\end{pmatrix} = \begin{pmatrix} (2^{\beta/2})^j & 0 \\ 0 & (2^{\beta/2})^{j/\beta}\end{pmatrix}
= \begin{pmatrix} \sigma^j & 0 \\ 0 & \sigma^{j\alpha}\end{pmatrix} = A^0_{\alpha,\sigma^j},
\end{align*}
and analogously $\tilde{A}_{\beta^{-1},2^{j\beta/2}}=A^1_{\alpha,\sigma^j}$. Furthermore, using $2^{j(\beta\pm 1)/2}=\sigma^{j(1\pm\alpha)}$, we obtain
\[
S_\ell A_{\beta^{-1},2^{j\beta/2}}=A_{\beta^{-1},2^{j\beta/2}}S_{\ell2^{(j(1-\beta))/2}}=A_{\beta^{-1},2^{j\beta/2}}S_{\ell \sigma^{-j(1-\alpha)}}=A^0_{\alpha,\sigma^j}S_{\ell\eta_j}=A^0_{\alpha,\sigma^j}S^0_{\ell,j}
\]
and
\[
S^T_\ell \tilde{A}_{\beta^{-1},2^{j\beta/2}}=\tilde{A}_{\beta^{-1},2^{j\beta/2}}S^T_{\ell2^{(j(1-\beta))/2}}=\tilde{A}_{\beta^{-1},2^{j\beta/2}}S^T_{\ell \sigma^{-j(1-\alpha)}}=A^1_{\alpha,\sigma^j}S^T_{\ell\eta_j}=A^1_{\alpha,\sigma^j}S^1_{\ell,j}.
\]
Taking into account $\tau=c$ and $2^{j(\beta+1)/4}=\sigma^{j(1+\alpha)/2}$, we obtain the following representation
\begin{align*}
\psi_{(0 , j, \ell , k)} &= \sigma^{j(1+\alpha)/2}\psi(A^0_{\alpha,\sigma^j}S^0_{\ell,j}\cdot-\tau k ), \\
\psi_{(1 , j, \ell , k)} &= \sigma^{j(1+\alpha)/2}\tilde{\psi}(A^1_{\alpha,\sigma^j}S^1_{\ell,j}\cdot-\tau k ).
\end{align*}
Therefore the system $SH(\phi,\psi,\tilde{\psi};c,\beta)=(\psi_\lambda)_{\lambda\in\Lambda^s}$ has the desired form with respect to the
generators given by $\gamma^{0}_{j,\ell,k}:=\psi$, $\gamma^{1}_{j,\ell,k}:=\tilde{\psi}$, and $\gamma^{0}_{-1,0,k}:=\sigma^{(1+\alpha)/2}\phi(A^0_{\alpha,\sigma}\cdot)$ for $j\ge0$,
$\ell\in\Z$ with $|\ell|\le \lceil 2^{j(\beta-1)/2} \rceil$, and $k\in\Z^2$.

It remains to prove that these generators satisfy \eqref{eq:shearcond}. We restrict our considerations to the functions $\gamma^{0}_{j,\ell,k}=\psi$.
The inverse Fourier transform of $\partial^{\rho} \hat{\psi}$, where $\rho\in\N_0^2$, is up to a constant given by $x\mapsto x^{\rho}\psi(x)$.
By smoothness and compact support of $\psi_1,\psi_2$, we find that for any $|\rho|\le L$ the functions
\[
x\mapsto\partial^{(N_1,N_1+N_2)}\big(x^{\rho}\psi(x)\big) \quad\text{ and }\quad x\mapsto x^{\rho}\psi(x)
\]
belong to $L^1(\R^2)$. Hence, the Fourier transforms
\[
 \xi\mapsto\xi_1^{N_1}\xi_2^{N_1+N_2}\partial^{\rho}\hat{\psi}(\xi) \quad\text{ and }\quad \xi\mapsto\partial^{\rho}\hat{\psi}(\xi)
\]
are continuous and contained in $L^\infty(\R^2)$. It follows that
\[
\langle\xi_1\rangle^{N_1}\langle\xi_2\rangle^{N_1+N_2}\partial^{\rho}\hat{\psi}(\xi)
\]
is bounded. Using
$
\langle x \rangle \langle y \rangle \ge \langle \sqrt{x^2+y^2} \rangle
$
we get the decay estimate for large frequencies
\[
|\partial^{\rho}\hat{\psi}(\xi)| \lesssim \langle |\xi| \rangle^{-N_1} \langle \xi_2 \rangle^{-N_2}.
\]

Let us turn to the moment conditions. Let $\rho=(\rho_1,\rho_2)\in\N_0^2$ with $|\rho_1|\le L$ for some $L=0,\ldots,M$. Then
\[
x^\rho\psi(x)=x_1^{\rho_1}\psi_1(x_1)x_2^{\rho_2}\psi_2(x_2)
\]
restricted to the variable $x_1$ possesses at least $M-L$ vanishing moments, since $\psi_1$ is assumed to possess $M$ vanishing moments.
This yields a decay of order $\min\{1,|\xi_1|^{M-L}\}$ for the derivatives up to order $L$ of $\hat{\psi}$ by the following lemma,
whose proof can be found, e.g., in \cite{Grohs2011}.

\begin{lemma}[\cite{Grohs2011}]\label{lem:momfreq}
Suppose that $g:\R\rightarrow\C$ is continuous, compactly supported and possesses $M$ vanishing moments. Then
\[
|\hat{g}(\xi)| \lesssim \min\{1,|\xi|\}^M.
\]
\end{lemma}

The proof is finished.

\subsection{Proof of Theorem \ref{thm:almostorth}}
\label{subsec:proofalmostorth}

We start by collecting some useful lemmata in Subsections~\ref{sec:estimates}, \ref{sec:basicestimates}, and \ref{subsec:diffop}, followed by the
actual proof of Theorem \ref{thm:almostorth} in Subsection~\ref{sec:almostorth}.

\subsubsection{General Estimates}\label{sec:estimates}

The following lemma can be found in \cite[Appendix K.1]{Grafakos2008}.

\begin{lemma}\label{lem:grafakos}
For $N>1$ and $a,a'\in \mathbb{R}_+$, we have the inequality
\[
\int_\mathbb{R}\left(1+a|x|\right)^{-N} \left(1+a'|x-y|\right)^{-N}dx\lesssim \max\{a,a'\}^{-1}(1 + \min\{a,a'\}|y|)^{-N}.
\]
\end{lemma}

The following result can be regarded as a corollary from the previous lemma.

\begin{lemma}\label{lem:bumps}
Assume that $|\theta| \le \frac{\pi}{2}$ and $N>1$. Then we have for $a,a'>0$ the inequality
\begin{equation}\label{eq:angulardecay}
\int_{\mathbb{T}}\left(1+a|\sin(\varphi)|\right)^{-N}\left(1+a'|\sin(\varphi+\theta)|\right)^{-N}d\varphi\lesssim \max\{a,a'\}^{-1}(1 + \min\{a,a'\}|\theta|)^{-N}.
\end{equation}
\end{lemma}

\begin{proof}
For $\varphi\in\mathbb{T}$, we have the estimate
\[
|\sin(\varphi)|\leq \left\{\begin{array}{cc}|\varphi| & \varphi \in I_1:=\left [ -\frac{\pi}{2},\frac{\pi}{2}\right],\\
|\varphi - \pi|&\varphi\in I_2:=\left [ \frac{\pi}{2},\pi\right],\\
|\varphi + \pi|&\varphi\in I_3:=\left [-\pi,- \frac{\pi}{2}\right].
\end{array}\right.
\]
In order to use Lemma \ref{lem:grafakos} we now split $\mathbb{T}$ into nine intervals depending on $\varphi + \theta, \varphi \in I_1,I_2,I_3$.
Then the left-hand side of (\ref{eq:angulardecay}) can be estimated by nine terms of the form
\[
\int_\mathbb{\mathbb{R}}\left(1+a|\varphi|\right)^{-N}\left(1+a'|\varphi+\vartheta+\theta|\right)^{-N}d\varphi,
\]
where $\vartheta\in \left\{0,\pm\pi,\pm 2\pi \right\}$. By Lemma \ref{lem:grafakos}, this expression can be bounded by a constant
times
\[
\max\{a,a'\}^{-1}(1 + \min\{a,a'\}|\theta + \vartheta|)^{-N}.
\]
Now it remains to note that for $\vartheta\in \left\{\pm \pi,\pm 2\pi\right\}$ and $|\theta|\le \frac{\pi}{2}$ we have $|\theta + \vartheta|\geq |\theta|$.
This proves the lemma.
\end{proof}

\subsubsection{Basic Estimates of $S_{\lambda,M,N_1,N_2}$}\label{sec:basicestimates}

We now consider the function $S_{\lambda,M,N_1,N_2} : \R^+_0 \times [0,2\pi) \to \R$ for $\lambda \in \Lambda$ \msch{and
$M, N_1, N_2 \in \N_0$ which} is defined in polar coordinates by
\[
S_{\lambda,M,N_1,N_2}(r,\varphi):= \min \left\{1,s_\lambda^{-1}(1+r)\right\}^{M}\left(1+s_\lambda^{(1-\alpha)}|
\sin(\varphi+\theta_\lambda)|\right)^{-N_2}\left(1 + s_\lambda^{-1}r\right)^{-N_1}.
\]
The reader might want to compare this definition with \eqref{eq:moldecaypolar2}.

The following lemma will be used in order to decouple the angular and the radial variables of this function.

\begin{lemma}\label{lem:polarest}
For every $0\le K \le N_2$,
\[
\min \left\{1,s_\lambda^{-1 }(1+r)\right\}^{M}\left(1 + s_\lambda^{-1}r\right)^{-N_1} \left(1 + s_\lambda^{-\alpha}r |\sin(\varphi + \theta_\lambda)|\right)^{-N_2}
\lesssim S_{\lambda,M-K,N_1,K}(r,\varphi).
\]
\end{lemma}

\begin{proof}
After choosing $K$, we can estimate the quantity on the left hand side by
\[
\min \left\{1,s_\lambda^{-1 }(1+r)\right\}^{M-K}\left(1 + s_\lambda^{-1}r\right)^{- N_1} \left(\frac{\min \left\{1,s_\lambda^{-1 }(1+r)\right\}}{1+s_\lambda^{-\alpha}r|
\sin(\varphi+\theta_\lambda)|}\right)^{K}.
\]
We need to show that
\begin{equation}\label{eq:freqest_help}
\frac{\min \left\{1,s_\lambda^{-1 }(1+r)\right\}}{1+s_\lambda^{-\alpha }r|\sin(\varphi+\theta_\lambda)|}\lesssim
\left(1 + s_\lambda^{(1-\alpha)}|\sin(\varphi + \theta_\lambda)|\right)^{-1}.
\end{equation}

In order to prove (\ref{eq:freqest_help}), we distinguish three cases:
\begin{itemize}
\item {\bf $ \mathbf{r \le 1}$: }For $r\le 1$ we have
\[
\frac{\min \left\{1,s_\lambda^{-1}(1+r)\right\}}{1+s_\lambda^{-\alpha }r| \sin(\varphi+\theta_\lambda)|}
\lesssim \min\left\{1,s_\lambda^{-1}\right\}
\lesssim \left(1 + s_\lambda^{(1-\alpha)}|\sin(\varphi + \theta_\lambda)|\right)^{-1}.
\]
\item {\bf $\mathbf{s_\lambda\le r}$: }  In this case we derive
\begin{eqnarray*}
\frac{\min \left\{1,s_\lambda^{-1}(1+r)\right\}}{1+s_\lambda^{-\alpha}r|\sin(\varphi+\theta_\lambda)|}
&= &\frac{1}{1+s_\lambda^{-\alpha}r|\sin(\varphi+\theta_\lambda)|}\le \frac{1}{1+s_\lambda^{-\alpha}s_\lambda|\sin(\varphi+\theta_\lambda)|}\\
&= &\left(1 + s_\lambda^{(1-\alpha)}|\sin(\varphi + \theta_\lambda)|\right)^{-1}.
\end{eqnarray*}
\end{itemize}

If $s_\lambda>1$ we have to examine a third case.
\begin{itemize}
\item {\bf $\mathbf{1<r<s_\lambda}$: }In this case we have
\[
\frac{\min \left\{1,s_\lambda^{-1}(1+r)\right\}}{1+s_\lambda^{-\alpha}r|\sin(\varphi+\theta_\lambda)|}
\leq\frac{s_\lambda^{-1}(1+r)}{1+s_\lambda^{-\alpha}r|\sin(\varphi+\theta_\lambda)|}
\leq\frac{1+r}{r}\frac{1}{\frac{s_\lambda}{r} + s_\lambda^{(1-\alpha)}|\sin(\varphi + \msch{\theta_\lambda})|}.
\]
\msch{Since $r>1$, we have $\frac{1+r}{r}< 2$, and since $r<s_\lambda$, also $\frac{s_\lambda}{r}>1$ holds.}
\end{itemize}

This proves the statement.
\end{proof}

The next lemma provides estimates for the inner product of two functions of the form $S_{\lambda,M,N_1,N_2}$.

\begin{lemma}\label{lem:freqangdec}
We assume $s_\lambda, s_\mu\ge c>0$ for all $\lambda\in\Lambda$ and $\mu\in\Delta$. For $A,B\ge1$ and
\[
N_1\geq A + \frac{1+\alpha}{2}, \quad N_2\geq B,\quad M > N_1-2,
\]
we have
\begin{eqnarray*}
\lefteqn{\hspace*{-1cm}(s_\lambda s_\mu)^{-\frac{1+\alpha}{2}}\int_{\mathbb{R}_+}\int_{\mathbb{T}}S_{\lambda,M,N_1,N_2}(r,\varphi)S_{\mu,M,N_1,N_2}(r,\varphi)rdr d\varphi}\\
&\hspace*{2cm} \lesssim& \max\Big\{ \frac{s_\lambda}{s_\mu},\frac{s_\mu}{s_\lambda} \Big\}^{-A}\left(1 + \min\{s_\lambda,s_\mu\}^{(1-\alpha)}|\theta_\lambda - \theta_{\mu}|\right)^{-B}.
\end{eqnarray*}
\end{lemma}

\begin{proof}
We assume that $s_{\mu}\geq s_\lambda$ and start by proving the angular decay. By Lemma \ref{lem:bumps} and $N_2\geq B\geq 1$,
\begin{eqnarray*}
(s_\lambda s_\mu)^{-\frac{1+\alpha}{2}}\int_{\mathbb{R}_+}\int_{\mathbb{T}}S_{\lambda,M,N_1,N_2}(r,\varphi)S_{\mu,M,N_1,N_2}(r,\varphi)rdr d\varphi
&\lesssim &\mathcal{S}\cdot (\frac{s_\mu}{s_\lambda})^{\frac{1+\alpha}{2}}\left(1+s_\lambda^{(1-\alpha)}|\theta_\lambda - \theta_{\mu}|\right)^{-B},
\end{eqnarray*}
where
\begin{equation*}
\mathcal{S}:=
s_\mu^{-2}\int_{\mathbb{R}_+}\min \left\{1,s_\lambda^{-1}(1+r)\right\}^{M}\min \left\{1,s_\mu^{-1}(1+r)\right\}^{M}
\left(1 + s_\lambda^{-1}r\right)^{-N_1}\left(1 + s_\mu^{-1}r\right)^{-N_1}rdr.
\end{equation*}

The remaining estimate
\begin{align}\label{eq:remainest}
\mathcal{S}\lesssim (s_\mu/s_\lambda)^{-\left(A+\frac{1+\alpha}{2}\right)}
\end{align}
is proved by splitting up the integral into the three parts $\mathcal{S}_i$, $i=1,2,3$, where the integration ranges over $0< r < 1$, $1\le r \le \max\{1,s_{\mu}\}$
and $\max\{1,s_{\mu}\}<r$, respectively.\\[1ex]
{\it Case 1} ($r< 1$):
For $\mathcal{S}_1$ we integrate over $0<r<1$. Here we use the moment property and $s_\lambda\ge c>0$ to estimate
\begin{eqnarray*}
\mathcal{S}_1
&\lesssim &s_\mu^{-2}\int_0^{1} s_\lambda^{-M}s_\mu^{-M} dr\\
& = &s_\mu^{-(2+M)}s_\lambda^{-M} \\
& \lesssim &s_\mu^{-(2+M)}s_\lambda^{M+2}\\
& =&(s_\mu/s_\lambda)^{-(M+2)} \\
&\le& (s_\mu/s_\lambda)^{-(A+\frac{1+\alpha}{2})}.
\end{eqnarray*}
{\it Case 2} ($1\le r \le \max\{1,s_{\mu}\}$):
If $s_\mu\le1$ then $\mathcal{S}_2=0$. For $s_\mu>1$ we estimate
\begin{eqnarray*}
\mathcal{S}_2
&\lesssim &s_\mu^{-2}\int_{1}^{s_{\mu}}\left(s_\mu^{-1}r\right)^M\left(s_\lambda^{-1}r\right)^{-N_1}rdr\\
&\le&s_\mu^{-(2+M)}s_\lambda^{N_1}\int_{0}^{s_{\mu}}r^{M+1-N_1}dr\\
&\lesssim &s_\mu^{-(2+M)}s_\lambda^{N_1}s_\mu^{(M+2-N_1)}\\
&=&(s_\mu/s_\lambda)^{-N_1}\\
&\le & (s_\mu/s_\lambda)^{-(A+\frac{1+\alpha}{2})}.
\end{eqnarray*}
{\it Case 3} ($\max\{1,s_{\mu}\}< r$):
For $\mathcal{S}_3$ we estimate
\begin{eqnarray*}
\mathcal{S}_3
&\lesssim &s_\mu^{-2}\int_{s_{\mu}}^\infty \left(s_\lambda^{-1}r\right)^{-N_1}\left(s_\mu^{-1}r\right)^{-N_1}r dr\\
&=&s_\mu^{-2}s_\mu^{N_1}s_\lambda^{N_1}\int_{s_{\mu}}^\infty r^{-2N_1+1}dr\\
&\lesssim&s_\mu^{-2}s_\mu^{N_1}s_\lambda^{N_1}s_\mu^{(-2N_1+2)}\\
&=&(s_\mu/s_\lambda)^{-N_1}\\
&\le& (s_\mu/s_\lambda)^{-(A+\frac{1+\alpha}{2})}.
\end{eqnarray*}

Altogether, this establishes \eqref{eq:remainest}.
\end{proof}

\subsubsection{Estimates with Differential Operator}\label{subsec:diffop}

Finally, we require some estimates \msch{of the symmetric differential operator $\mathcal{L}$ (acting on the frequency variable $\xi$) }  defined by
\[
\mathcal{L} := I - s_0^{2\alpha}\Delta_\xi - \frac{s_0^{2}}{1 + s_0^{2(1-\alpha)}|\delta \theta|^2} \frac{\partial^2}{\partial \xi_1^2},
\]
which will be given by the second lemma. The first lemma will be required within its proof.

\begin{lemma}\label{lem:prodrule}
\msch{Given two functions $a^{(\lambda)}, b^{(\mu)}$ satisfying (\ref{eq:molcond}) for $L,M,N_1,N_2$, the expression}
\[
\mathcal{L}\left(\hat a^{(\lambda)}\left(A_{\alpha,s_\lambda^{-1}}R_{\theta_\lambda}\xi \right) \overline{\hat b^{(\mu)} \left(A_{\alpha,s_\mu^{-1}}R_{\theta_\mu}\xi\right)}\right)
\]
can be written as a finite linear combination of terms of the form
\[
\hat c^{(\lambda)}\left(A_{\alpha,s_\lambda^{-1}}R_{\theta_\lambda}\xi \right) \overline{\hat d^{(\mu)} \left(A_{\alpha,s_\mu^{-1}}R_{\theta_\mu}\xi\right)}
\]
with $c,d$ satisfying (\ref{eq:molcond}) for $L-2,M,N_1,N_2$.
\end{lemma}

\begin{proof}
To prove the claim we treat the three summands of the operator $\mathcal{L}$ separately. The first part is the identity, and therefore
the statement is trivial. To handle the second part, the frequency Laplacian $s_0^{2\alpha }\Delta$, we use the product rule
\[
\Delta(fg) = 2\left(\partial^{(1,0)}f\partial^{(1,0)}g + \partial^{(0,1)}f\partial^{(0,1)}g\right) + (\Delta f)g + f(\Delta g).
\]
Therefore we need to estimate the derivatives of degree $1$ and the Laplacians of the two factors in the product
\[
\hat a^{(\lambda)}\left(A_{\alpha,s_\lambda^{-1}}R_{\theta_\lambda}\xi \right)\overline{\hat b^{(\mu)}\left(A_{\alpha,s_\mu^{-1}}R_{\theta_\mu}\xi\right)}=:A(\xi)B(\xi).
\]

For this, we start with the first factor,
\[
A(\xi)=\hat a^{(\lambda)}\left(s_\lambda^{-1}\cos(\theta_\lambda)\xi_1- s_\lambda^{-1}\sin(\theta_\lambda)\xi_2 ,
s_\lambda^{-\alpha }\sin(\theta_\lambda)\xi_1+ s_\lambda^{-\alpha }\cos(\theta_\lambda)\xi_2\right).
\]
Set
\[
A_1(\xi):=\partial^{(1,0)}\hat a^{(\lambda)}\left(A_{\alpha,s_\lambda^{-1}}R_{\theta_\lambda}\xi\right)\ \mbox{ and }\
 A_2(\xi):=\partial^{(0,1)}\hat a^{(\lambda)}\left(A_{\alpha,s_\lambda^{-1}}R_{\theta_\lambda}\xi\right).
\]
By definition, the functions $A_1,\ A_2$ satisfy (\ref{eq:molcond}) with $L$ replaced by $L-1$. An application of the chain rule shows that
\[
\partial^{(1,0)}A(\xi)=s_\lambda^{-1}\cos(\theta_\lambda)A_1(\xi)+ s_\lambda^{-\alpha }\sin(\theta_\lambda)A_2(\xi).
\]
Analogously, one can compute
\[
\partial^{(0,1)}A(\xi)=-s_\lambda^{-1}\sin(\theta_\lambda)A_1(\xi)+ s_\lambda^{-\alpha }\cos(\theta_\lambda)A_2(\xi),
\]
and the exact same expressions for $B$ using the obvious definitions for $B_1$ and $B_2$. We get
\begin{eqnarray*}
\partial^{(1,0)}A\partial^{(1,0)}B
&=&(s_\lambda s_\mu)^{-1}\cos(\theta_\lambda)\cos(\theta_\mu)A_1B_1+s_\lambda^{-\alpha}s_\mu^{-1}\sin(\theta_\lambda)\cos(\theta_\mu)A_2B_1\\
&&+s_\mu^{-\alpha}s_\lambda^{-1}\sin(\theta_\mu)\cos(\theta_\lambda)A_1B_2+ (s_\lambda s_\mu)^{-\alpha}\sin(\theta_\lambda)\sin(\theta_\mu)A_2B_2.
\end{eqnarray*}
It follows that $s_0^{2\alpha }\partial^{(1,0)}A\partial^{(1,0)}B$ can be written as a linear combination as claimed (recall that $s_0 = \min\{s_\lambda,s_\mu\})$.
The same argument applies to the product $s_0^{2\alpha }\partial^{(0,1)}A\partial^{(0,1)}B$.

It remains to consider the factor
\[
(\Delta A)B + A(\Delta B),
\]
where, for symmetry reasons, we only treat the summand $(\Delta A)B$. In fact, it suffices to only consider
\[
(\partial^{(2,0)}A)B
 =\left(s_\lambda^{-2}\cos(\theta_\lambda)^2A_{11}+2s_\lambda^{-(1+\alpha)}\sin(\theta_\lambda)\cos(\theta_\lambda)A_{12} - s_\lambda^{-2\alpha }\sin(\theta_\lambda)^2A_{22}\right)B
\]
with $A_{ij}$ defined in an obvious way, satisfying (\ref{eq:molcond}) with $L$ replaced by $L-2$. The term $(\partial^{(0,2)}A)B$, and hence
$(\Delta A)B$, can be handled in the same way, as can $A(\Delta B)$. This takes care of the term $s_0^{2\alpha }\Delta$ in the definition of
$\mathcal{L}$.

Finally, we need to handle the last term in the definition of $\mathcal{L}$, namely
\[
\frac{s_0^{2}}{1 + s_0^{2(1-\alpha)}| \theta_\mu|^2} \frac{\partial^2}{\partial \xi_1^2}
\]
for $\theta_\lambda = 0$ (otherwise the second order derivative would be in the direction of the unit vector with angle $\theta_\lambda$ with
obvious modifications in the proof). With our notation and using the product rule we need to consider terms of the form
\[
\left(\partial^{(2,0)}A\right) B,\quad \left(\partial^{(1,0)}A\right)\left(\partial^{(1,0)} B\right),\quad A\left(\partial^{(2,0)} B\right),
\]
and show that each of them, multiplied by the factor $s_0^{2}/(1 + s_0^{2(1-\alpha)}|\theta_\mu|^2)$, satisfies the desired representation.

Let us start with $\left(\partial^{(2,0)}A\right) B$, which, using the fact that $\sin(\theta_\lambda) = 0$, can be written as
\[
\left(\partial^{(2,0)}A\right) B = s_\lambda^{-2}A_{11}B,
\]
and which clearly satisfies the desired assertion.

Now consider the expression $\left(\partial^{(1,0)}A\right)\left(\partial^{(1,0)} B\right)$,
which can be written as
\[
\left(\partial^{(1,0)}A\right)\left(\partial^{(1,0)} B\right) = s_\lambda^{-1}s_\mu^{-1}\cos(\theta_\mu)A_1B_1+ s_\lambda^{-1}s_\mu^{-\alpha }\sin(\theta_\mu)A_1B_2.
\]
The first summand in this expression clearly causes no problems. To handle the second term we need to show that
\begin{equation}\label{eq:diffopbound}
\frac{s_0^{2}}{1 + s_0^{2(1-\alpha)}|\theta_\mu|^2} s_\lambda^{-1}s_\mu^{-\alpha }\sin(\theta_\mu) \lesssim 1.
\end{equation}
Here we have to distinguish two cases. First, assume that $|\theta_\mu|\le s_0^{-(1-\alpha)}$. Then we can estimate
$\sin(\theta_\mu)\lesssim s_0^{-(1-\alpha)}$, which readily yields the desired bound for (\ref{eq:diffopbound}).
For the case $|\theta_\mu|\geq s_0^{-(1-\alpha) }$ we estimate
\[
\frac{s_0^{2}}{1 + s_0^{2(1-\alpha)}|\theta_\mu|^2}s_\lambda^{-1}s_\mu^{-\alpha }\sin(\theta_\mu)
\lesssim\frac{s_0^{2}}{1 + s_0^{(1-\alpha)}|\theta_\mu|}s_0^{-1}s_0^{-\alpha }|\theta_\mu|
\le\frac{s_0^{2}}{s_0^{(1-\alpha)}|\theta_\mu|}s_0^{-1}s_0^{-\alpha }|\theta_\mu| = 1
\]
which proves (\ref{eq:diffopbound}) also for this case.

We are left with estimating the term $A\left(\partial^{(2,0)} B\right)$, which can be written as
\[
s_\mu^{-2}\cos(\theta_\mu)^2 A B_{11}+ 2s_\mu^{-(1+\alpha)}\sin(\theta_\mu)\cos(\theta_\mu)AB_{12}+ s_\mu^{-2\alpha }\sin(\theta_\mu)^2A B_{22}.
\]
The first two terms are of a form already treated, and the last term can be handled using the fact that $\sin(\theta_\mu)^2\le \theta_\mu^2$.
\end{proof}

\begin{lemma}\label{lem:diffopprod}
Assume that the assumptions of Theorem~\ref{thm:almostorth} hold for two systems of $\alpha$-molecules of order $(L,M,N_1,N_2)$ \ms{with
respective generating functions $(a^{(\lambda)})_{\lambda}$ and $(b^{(\mu)})_\mu$}. Then we have
\[
\mathcal{L}^k\left(\hat a^{(\lambda)}\left(A_{\alpha,s_\lambda^{-1}}R_{\theta_\lambda}\xi \right)\overline{\hat b^{(\mu)}\left(A_{\alpha,s_\mu^{-1}}R_{\theta_\mu}\xi\right)}\right)
\lesssim S_{\lambda,M-N_2,N_1,N_2}(\xi)S_{\mu,M-N_2,N_1,N_2}(\xi)
\]
for all $k\le L/2$.
\end{lemma}

\begin{proof}
We show that
\begin{eqnarray}\nonumber
\lefteqn{\left|\mathcal{L}^k\left(\hat a^{(\lambda)}\left(A_{\alpha,s_\lambda^{-1}}R_{\theta_\lambda} \xi \right)\overline{\hat b^{(\mu)}
\left(A_{\alpha,s_\mu^{-1}}R_{\theta_\mu}\xi\right)}\right)\right|}\\ \nonumber
& \lesssim & \min \left\{1,s_\lambda^{-1}(1+r)\right\}^{M} \left(1 + s_\lambda^{-1}r \right)^{-N_1} \left(1 + s_\lambda^{-\alpha}r |\sin(\varphi + \theta_\lambda)|\right)^{-N_2}\\
\label{eq:product}
& & \cdot \min \left\{1,s_\mu^{-1 }(1+r)\right\}^{M} \left(1 + s_\mu^{-1}r \right)^{-N_1} \left(1 + s_\mu^{-\alpha }r |\sin(\varphi + \theta_\mu)|\right)^{-N_2}
\end{eqnarray}
which, using Lemma \ref{lem:polarest} with $K=N_2$, implies the desired statement.

To prove (\ref{eq:product}), we use induction in $k$, namely we show that if we have two functions $a^{(\lambda)}, b^{(\mu)}$ satisfying
(\ref{eq:molcond}) for $L,M,N_1,N_2$, then the expression
\[
\mathcal{L}\left(\hat a^{(\lambda)}\left(A_{\alpha,s_\lambda^{-1}}R_{\theta_\lambda}\xi \right)\overline{\hat b^{(\mu)}\left(A_{\alpha,s_\mu^{-1}}R_{\theta_\mu}\xi\right)}\right)
\]
can be written as a finite linear combination of terms of the form
\[
\hat c^{(\lambda)}\left(A_{\alpha,s_\lambda^{-1}}R_{\theta_\lambda}\xi \right) \overline{\hat d^{(\mu)}\left(A_{\alpha,s_\mu^{-1}}R_{\theta_\mu}\xi\right)}
\]
with $c,d$ satisfying (\ref{eq:molcond}) and $L$ replaced by $L-2$, see Lemma \ref{lem:prodrule}. Iterating this argument we can establish that for $k\le L/2$
\begin{equation}\label{eq:estt}
\mathcal{L}^k\left(\hat a^{(\lambda)}\left(A_{\alpha,s_\lambda^{-1}}R_{\theta_\lambda}\xi \right)\overline{\hat b^{(\mu)}\left(A_{\alpha,s_\mu^{-1}}R_{\theta_\mu}\xi\right)}\right)
\end{equation}
can be expressed as a finite linear combination of terms of the form
\begin{equation}\label{eq:product1}
\hat c^{(\lambda)}\left(A_{\alpha,s_\lambda^{-1}}R_{\theta_\lambda}\xi \right) \overline{\hat d^{(\mu)}\left(A_{\alpha,s_\mu^{-1}}R_{\theta_\mu}\xi\right)}
\end{equation}
with
\begin{equation}\label{eq:product2}
\left| \hat c^{(\lambda)}(\xi)\right|
\lesssim \min\left\{1,s_\lambda^{-1} + |\xi_1| + s_\lambda^{-(1-\alpha)}|\xi_2|\right\}^M\left\langle |\xi|\right\rangle^{-N_1} \langle \xi_2 \rangle^{-N_2},
\end{equation}
and an analogous estimate for $d^{(\mu)}$. Combining (\ref{eq:product1}) and (\ref{eq:product2}), we obtain that $|(\ref{eq:estt})|$
can --~up to a constant~-- be upperbounded by the product of
\[
\min \left\{1,s_\lambda^{-1}+\left|\left(A_{\alpha,s_\lambda^{-1}}R_{\theta_\lambda}\xi\right)_1 \right|
+s_\lambda^{-(1-\alpha) }\left|\left(A_{\alpha,s_\lambda^{-1}}R_{\theta_\lambda}\xi \right)_2 \right|\right\}^M
\left\langle \left |A_{\alpha,s_\lambda^{-1}}R_{\theta_\lambda}\xi\right|\right\rangle^{-N_1}
 \msch{\left\langle \left(A_{\alpha,s_\lambda^{-1}}R_{\theta_\lambda}\xi \right)_2 \right\rangle^{-N_2}}
\]
and
\[\min \left\{1,s_\mu^{-1}+\left|\left(A_{\alpha,s_\mu^{-1}}R_{\theta_\mu}\xi\right)_1 \right|
+s_\mu^{-(1-\alpha) }\left|\left(A_{\alpha,s_\mu^{-1}}R_{\theta_\mu}\xi \right)_2 \right|\right\}^M \left\langle \left |A_{\alpha,s_\mu^{-1}}R_{\theta_\mu}\xi\right |\right\rangle^{-N_1}
 \msch{\left\langle \left(A_{\alpha,s_\mu^{-1}}R_{\theta_\mu}\xi \right)_2 \right\rangle^{-N_2}}.
\]
Transforming this inequality into polar coordinates as in (\ref{eq:moldecaypolar2}) yields (\ref{eq:product}). This finishes the proof.
\end{proof}

\subsubsection{Actual Proof}\label{sec:almostorth}

We now have all the ingredients to prove Theorem \ref{thm:almostorth}. By our assumptions on $(L,M,N_1,N_2)$, there exist $\widetilde{N}_1$ and
$\widetilde{N}_2$ such that $N_1\ge\widetilde{N}_1\ge N+\frac{1+\alpha}{2}$ and $N_2\ge\widetilde{N}_2\ge N+\frac{1+\alpha}{2}$ and
$M> \widetilde{N}_1+\widetilde{N}_2-2$. The systems $(m_\lambda)_\lambda$ and $(p_\mu)_\mu$ are also $\alpha$-molecules of order
$(L,M,\widetilde{N}_1,\widetilde{N}_2)$, satisfying the assumptions of the Theorem. Thus, we can without loss of generality assume the additional
condition $M> N_1+N_2-2$.

To keep the notation simple, we assume that $\theta_\lambda = 0$ and define $s_0:=\min\{s_\lambda , s_{\mu}\}$. Further, we set
\[
\delta x := x_\lambda - x_\mu,\quad \delta \theta:= \theta_\lambda - \theta_\mu.
\]
By definition, we can write
\[
m_\lambda(\cdot) = s_\lambda^{\frac{1+\alpha}{2}}a^{(\lambda)}\left(A_{\alpha,s_\lambda}R_{\theta_\lambda}(\cdot - x_\lambda)\right),
\quad
p_\mu(\cdot) = s_\mu^{\frac{1+\alpha}{2}}b^{(\mu)}\left(A_{\alpha,s_\mu}R_{\theta_\mu} (\cdot - x_\mu)\right),
\]
where both $a^{(\lambda)}$ and $b^{(\mu)}$ satisfy (\ref{eq:molcond}). We have the equality
\begin{eqnarray}\nonumber
\left\langle m_\lambda,p_{\mu}\right\rangle \hspace*{-0.25cm}
& = & \hspace*{-0.25cm} \left\langle \hat m_\lambda,\hat p_{\mu}\right\rangle\\
& = &\hspace*{-0.25cm} (s_\lambda s_\mu)^{-\frac{1+\alpha}{2}}\int_{\mathbb{R}^2} \hat a^{(\lambda)}\left(A_{\alpha,s_\lambda^{-1}}R_{\theta_\lambda}\xi \right) \overline{\hat b^{(\mu)}
\left(A_{\alpha,s_\mu^{-1}}R_{\theta_\mu}\xi\right)}\exp\left(-2\pi i \xi \cdot \delta x\right) d\xi \nonumber \\ \label{eq:parint1}
& = &  \hspace*{-0.25cm} (s_\lambda s_\mu)^{-\frac{1+\alpha}{2}}\int_{\mathbb{R}^2}\mathcal{L}^k\left(\hat a^{(\lambda)}\left(A_{\alpha,s_\lambda^{-1}}R_{\theta_\lambda}\xi \right) \overline{\hat b^{(\mu)}
\left(A_{\alpha,s_\mu^{-1}}R_{\theta_\mu}\xi\right)}\right)\mathcal{L}^{-k}\left(\exp\left(-2\pi i \xi \cdot \delta x\right)\right)d\xi,
\end{eqnarray}
where $\mathcal{L}$ is the symmetric differential operator (acting on the frequency variable) defined by
\[
\mathcal{L} := I - s_0^{2\alpha}\Delta_\xi - \frac{s_0^{2}}{1 + s_0^{2(1-\alpha)}|\delta \theta|^2} \frac{\partial^2}{\partial \xi_1^2}.
\]
We have
\begin{equation}\label{eq:parint2}
\mathcal{L}^{-k}\left(\exp\left(-2\pi i \xi \cdot \delta x\right)\right)
= \left(1 + s_0^{2\alpha }|\delta x|^2 + \frac{s_0^{2}}{1 + s_0^{2(1-\alpha)}|\delta \theta|}\langle e_\lambda , \delta x\rangle^2\right)^{-k}\exp\left(-2\pi i \xi \cdot \delta x\right),
\end{equation}
where $e_\lambda$ denotes the unit vector pointing in the direction described by the angle $\theta_\lambda$.
By Lemma \ref{lem:diffopprod} and for $k\le \frac{L}{2}$, we have the inequality
\[
\mathcal{L}^k\left(\hat a^{(\lambda)}\left(A_{\alpha,s_\lambda^{-1}}R_{\theta_\lambda}\xi \right)\overline{\hat b^{(\mu)}
\left(A_{\alpha,s_\mu^{-1}}R_{\theta_\mu}\xi\right)}\right)\lesssim S_{\lambda,M-N_2,N_1,N_2}(\xi)S_{\mu,M-N_2,N_1,N_2}(\xi).
\]
Then, by (\ref{eq:parint1}) and (\ref{eq:parint2}) it follows that
\begin{eqnarray*}
\lefteqn{\left|\langle m_\lambda , p_\mu \rangle \right|}\\
& \lesssim & (s_\lambda s_\mu)^{-\frac{1+\alpha}{2}}\hspace*{-0.2cm}\int_{\mathbb{R}^2}S_{\lambda,M-N_2,N_1,N_2}(\xi)S_{\mu,M-N_2,N_1,N_2}(\xi)d\xi
\left(1 + s_0^{2\alpha }|\delta x|^2 + \frac{s_0^{2}}{1 + s_0^{2(1-\alpha)}|\delta \theta|}\langle e_\lambda , \delta x\rangle^2\right)^{-k}
\end{eqnarray*}
for all $k\le \frac{L}{2}$. Now we can use Lemma \ref{lem:freqangdec} and the fact that $L\geq 2N$ to establish that
\begin{eqnarray*}
\left|\langle m_\lambda , p_\mu \rangle \right|
&\lesssim & \max\Big\{\frac{s_\lambda}{s_\mu},\frac{s_\lambda}{s_\mu}\Big\}^{-N}\left(1 + s_0^{2(1-\alpha)}|\delta\theta|^2 \right)^{-N}
\left(1 + s_0^{2\alpha }|\delta x|^2 + \frac{s_0^{2}}{1 + s_0^{2(1-\alpha)}|\delta \theta|}\langle e_\lambda , \delta x\rangle^2\right)^{-N}\\
&\le & \max\Big\{\frac{s_\lambda}{s_\mu},\frac{s_\lambda}{s_\mu}\Big\}^{-N}\left(1 + s_0^{2(1-\alpha)}|\delta\theta|^2+
s_0^{2\alpha }|\delta x|^2 + \frac{s_0^{2}}{1 + s_0^{2(1-\alpha)}|\delta \theta|}\langle e_\lambda , \delta x\rangle^2\right)^{-N}\\
& = &\omega_{\alpha}(\lambda,\mu)^{-N}.
\end{eqnarray*}
This proves the desired statement.

\subsection{Proofs of Section \ref{sec:approx}}

\subsubsection{Proof of Lemma \ref{lem:decayapprox}} \label{subsec:proofdecayapprox}

Let $(c^\ast_n)_{n\in\N}$ be a non-increasing rearrangement of the expansion coefficients $(c_\lambda)_\lambda\in \omega\ell^{2/(p+1)}(\Lambda)$ and
let $(m^\ast_n)_{n\in\N}$ be the accordingly reordered frame.
Then because of
\[
\sup_{n>0} n^{\frac{1}{p}} |c^\ast_n| \le \|(c_\lambda)_\lambda\|_{\omega\ell^p},
\]
we have $|c^*_n|\lesssim n^{-\frac{p+1}{2}}$, or equivalently $|c^*_n|^2\lesssim n^{-p-1}$,
for $n\in\N$. Summation yields
\[
\sum_{n=N+1}^\infty |c^*_n|^2 \lesssim \sum_{n=N+1}^\infty n^{-p-1}\le \int_{N}^\infty x^{-p-1}\,dx
=\frac{1}{p} N^{-p} \lesssim N^{-p}.
\]
Using the frame properties of $(m_\lambda)_{\lambda\in \Lambda}$, we conclude for the $N$-term approximation $f_N$ obtained by keeping the $N$ largest coefficients
\[
\|f-f_N\|_2^2 = \Big\|\sum_{n=N+1}^\infty c^*_nm^*_n\Big\|_2^2  \lesssim
\sum_{n=N+1}^\infty |c^*_n|^2 \lesssim N^{-p}.
\]

\subsubsection{Proof of Lemma \ref{lem:auxest}} \label{subsec:proofauxest}

(i): We start with part (i). For this, let $\sigma>1$, $\tau>0$, $(\omega_j)_j$ and $(L_j)_j$ be the parameters associated with
the parametrization $(\Lambda^c,\Phi^c)$. Further, let us put $s_0=\min\{s_\lambda,s_\mu\}$ to simplify the notation in the
proof. We now need to estimate the sum
\[
S=\sum_{\substack{\lambda\in\Lambda^c\\ s_\lambda \text{ fixed}}} (1+d_\alpha(\lambda,\mu))^{-N},
\]
taken over all curvelet indices $(j,\ell,k)\in\Lambda^c$ at scale $s_\lambda=\sigma^j$, where $j\in\N_0$ is fixed.

We begin with the estimate
\begin{align*}
(1+d_\alpha(\lambda,\mu))^{-N}
&=
\big(1+s_0^{2(1-\alpha)}|\theta_\lambda - \theta_{\mu}|^2 + s_0^{2\alpha}|x_\lambda
         - x_{\mu}|^2 +
        \frac{s_0^{2}|\langle e_\lambda , x_\lambda - x_{\mu}\rangle|^2}{1+s_0^{2(1-\alpha)}|\theta_\lambda-\theta_\mu|^2}\big)^{-N} \\
&=2^N
\big(2+2s_0^{2(1-\alpha)}|\theta_\lambda - \theta_{\mu}|^2 + 2s_0^{2\alpha}|x_\lambda
         - x_{\mu}|^2 +
        2\frac{s_0^{2}|\langle e_\lambda , x_\lambda - x_{\mu}\rangle|^2}{1+s_0^{2(1-\alpha)}|\theta_\lambda-\theta_\mu|^2}\big)^{-N} \\
&\le2^N
\big(2+2s_0^{2(1-\alpha)}|\theta_\lambda - \theta_{\mu}|^2 + s_0^{2\alpha}|x_\lambda
         - x_{\mu}|^2 +
        \frac{s_0^{2}|\langle e_\lambda , x_\lambda - x_{\mu}\rangle|^2}{1+s_0^{2(1-\alpha)}|\theta_\lambda-\theta_\mu|^2}\big)^{-N} \\
&\le2^N
\big(1+s_0^{2(1-\alpha)}|\theta_\lambda - \theta_{\mu}|^2 + s_0^{2\alpha}|x_\lambda
         - x_{\mu}|^2 + 2s_0|\langle e_\lambda , x_\lambda - x_{\mu}\rangle|\big)^{-N} \\
&\le2^N
\big(1+s_0^{2(1-\alpha)}|\theta_\lambda - \theta_{\mu}|^2 + s_0^{2\alpha}|x_\lambda
         - x_{\mu}|^2 + s_0|\langle e_\lambda , x_\lambda - x_{\mu}\rangle|\big)^{-N},  \\
\end{align*}
where we used the inequality between the arithmetic and the geometric mean
\[
\big(1+s_0^{2(1-\alpha)}|\theta_\lambda - \theta_{\mu}|^2\big) +
\frac{s_0^{2}}{1+s_0^{2(1-\alpha)}|\theta_\lambda-\theta_\mu|^2}|\langle e_\lambda , x_\lambda - x_{\mu}\rangle|^2 \ge 2\cdot s_0 |\langle e_\lambda , x_\lambda - x_{\mu}\rangle|.
\]
Denoting the components of a vector $z\in\R^2$ by $[z]_1$ and $[z]_2$, respectively, we further obtain
\begin{align}\label{eq:estterm1}
|\langle e_\lambda , x_\lambda - x_{\mu}\rangle|=|\langle R_{-\theta_\lambda}e_1 , R_{-\theta_\lambda}A^{-1}_{\alpha,s_\lambda}k - x_{\mu}\rangle|
=|\langle e_1 , A^{-1}_{\alpha,s_\lambda}k - R_{\theta_\lambda}x_{\mu}\rangle|=| s_\lambda^{-1}k_1 - [R_{\theta_\lambda}x_{\mu}]_1|,
\end{align}
where $e_1$ is the first unit vector of $\R^2$, and
\[
| x_\lambda - x_{\mu}|=| R_{-\theta_\lambda}A^{-1}_{\alpha,s_\lambda}k - x_{\mu}|
=| A^{-1}_{\alpha,s_\lambda}k - R_{\theta_\lambda}x_{\mu}|\ge | s_\lambda^{-\alpha}k_2 - [R_{\theta_\lambda}x_{\mu}]_2| .
\]
By assumption the angles $\omega_j$ satisfy $\omega_j\asymp \sigma^{-j(1-\alpha)}=s_\lambda^{-(1-\alpha)}$. It follows
\begin{align}\label{eq:estterm3}
|\theta_\lambda-\theta_\mu|=|\ell\omega_j - \theta_\mu|=|\omega_j| |\ell - \theta_\mu/\omega_j|\asymp  s_\lambda^{-(1-\alpha)} |\ell - \theta_\mu/\omega_j|.
\end{align}
Altogether we deduce from $\omega_j\asymp s_\lambda^{-(1-\alpha)}$ and \eqref{eq:estterm1}-\eqref{eq:estterm3}
{\allowdisplaybreaks
\begin{align}\label{eq:standsum}
S &\lesssim \sum_{k\in\Z^2}\sum_{\ell=-L_j}^{L_j} \big(1+s_0^{2(1-\alpha)}|\theta_\lambda - \theta_{\mu}|^2 + s_0^{2\alpha}|x_\lambda
         - x_{\mu}|^2 + s_0|\langle e_\lambda , x_\lambda - x_{\mu}\rangle|\big)^{-N} \notag\\
&\lesssim \sum_{k\in\Z^2}\sum_{\ell=-L_j}^{L_j}
\big(1+(s_\lambda/s_0)^{-2(1-\alpha)}| \ell - \theta_\mu/\omega_j |^2 + s_0^{2\alpha}| s_\lambda^{-\alpha}k_2-[R_{\theta_\lambda}x_\mu]_2 |^2 +
        s_0 |s_\lambda^{-1}k_1-[R_{\theta_\lambda}x_\mu]_1 | \big)^{-N} \notag\\
&\lesssim \sum_{k\in\Z^2}\sum_{\ell\in\Z}
\big(1+| \ell\cdot (s_\lambda/s_0)^{-(1-\alpha)}-a_1 |^2 + |(s_\lambda/s_0)^{-\alpha}k_2-a_2(\ell) |^2 +  | (s_\lambda/s_0)^{-1}k_1-a_3(\ell) | \big)^{-N}
\end{align}
}
where $\theta_\lambda=\ell\omega_j$, $s_\lambda=\sigma^j$,  and the quantities
\begin{align*}
a_1:=s_0^{(1-\alpha)}\theta_\mu, && a_2(\ell):=s_0^{\alpha}[R_{\theta_\lambda}x_\mu]_2, && a_3(\ell):=s_0[R_{\theta_\lambda}x_\mu]_1
\end{align*}
depend on $j$ and $\mu$, and $a_2$ and $a_3$ also on $\ell$ as indicated by the notation.

To proceed, we distinguish the cases $s_\lambda\ge s_\mu$ and $s_\lambda<s_\mu$.
If $s_\lambda<s_\mu$ then $s_0=s_\lambda$ and the sum becomes
\begin{align*}
\sum_{k\in\Z^2}\sum_{\ell\in\Z}
\big(1+| \ell -a_1 |^2 + |k_2-a_2(\ell) |^2 +  | k_1-a_3(\ell) | \big)^{-N}.
\end{align*}
Since $N>2$, this expression is bounded by the constant
\[
C:=8 \sum_{k\in\N^2_0}\sum_{\ell\in\N_0}
\big(1+| \ell |^2 + |k_2 |^2 +  | k_1 | \big)^{-N}<\infty.
\]
In the other case, if $s_\lambda\ge s_\mu$, we have $s_0=s_\mu$ and
the sum can be interpreted as a Riemann sum, which is bounded up to a multiplicative constant by
the corresponding integral
\beqn
S&\lesssim& \sum_{k\in\Z^2}\sum_{\ell\in\Z}
\big(1+| \ell\cdot (s_\lambda/s_\mu)^{-(1-\alpha)}-a_1 |^2 + |(s_\lambda/s_\mu)^{-\alpha}k_2 -a_2(\ell)|^2 +  | (s_\lambda/s_\mu)^{-1}k_1-a_3(\ell) | \big)^{-N} \\
&=& (s_\lambda/s_\mu)^{2} \cdot \sum_{\ell\in\Z} (s_\lambda/s_\mu)^{-(1-\alpha)} \sum_{k_1\in\Z} (s_\lambda/s_\mu)^{-1}  \sum_{k_2\in\Z} (s_\lambda/s_\mu)^{-\alpha}
\big(1+| \ell\cdot (s_\lambda/s_\mu)^{-(1-\alpha)}-a_1 |^2 \\
&&+ |(s_\lambda/s_\mu)^{-\alpha}k_2 -a_2(\ell)|^2 +  | (s_\lambda/s_\mu)^{-1}k_1-a_3(\ell) | \big)^{-N} \\
&\lesssim& \max\Big\{\frac{s_\lambda}{s_\mu},1\Big\}^{2} \cdot \int_{\R}\,dy \int_{\R^2}\,dx
\big( 1+ | y |^2 + | x_2 |^2 + | x_1 | \big)^{-N}.
\eeqn
Precisely for $N>2$ the integral is finite. Further, the implicit constant ist independent of $s_\lambda$ and $\mu=(x_\mu, \theta_\mu, s_\mu)$.
This finishes the proof of part (i).

\vspace*{0.3cm}

(ii): We now turn to part (ii). Some arguments will be similar to part (i), which we will point out in the sequel. However,
often the utilization of shearing instead of rotation will require a different technical treatment, in particular,
due to the splitting into two parameter sets depending on the parameter $\varepsilon$.

Similar to the curvelet parametrization, the shearlet parametrization $(\Lambda^s,\Phi^s)$ is specified by a set of
parameters $\sigma>1$, $\tau>0$, $(\eta_j)_j$ and $(L_j)_j$. Further, we set $s_0=\min\{s_\lambda,s_\mu\}$ and let $j\in\N_0$ be
the fixed number  with $s_\lambda=\sigma^j$. The sum
\[
\sum_{\substack{\lambda\in\Lambda^s \\ s_\lambda\text{ fixed}}} (1+d_\alpha(\lambda,\mu))^{-N}
\]
can be split into two parts for $\varepsilon=0$ and $\varepsilon=1$. For symmetry reasons, both partial sums can be treated
in the same fashion and it therefore suffices to give the estimate for the part where $\varepsilon=0$.

We know from the proof of part (i) that
\beqn
(1+d_\alpha(\lambda,\mu))^{-N}
\lesssim
 \big(1+s_0^{2(1-\alpha)}|\theta_\lambda - \theta_{\mu}|^2 + s_0^{2\alpha}|x_\lambda
         - x_{\mu}|^2 + s_0|\langle e_\lambda , x_\lambda - x_{\mu}\rangle|\big)^{-N}.
\eeqn
Since $|\ell|\lesssim \sigma^{j(1-\alpha)}$ and $\eta_j\asymp \sigma^{-j(1-\alpha)}$ we have $|\ell\eta_j|\lesssim1$. Hence, there is a bound $B>0$ such that
\begin{align}\label{eq:shearbound}
|\ell\eta_j|\le B \text{ for all } j\in\N_0, |\ell|\le L_j.
\end{align}
In the proof of Proposition~\ref{prop:shearletmol} we have shown that the `transfer matrix' $T=R_{\theta_\lambda}(S_{\ell,j})^{-1}$
with $\theta_\lambda=\arctan(-\ell\eta_j)$ has the form
\begin{align}\label{eq:transfer}
T=\begin{pmatrix} \cos\theta_\lambda & 0 \\ \sin\theta_\lambda & \cos(\theta_\lambda)^{-1}  \end{pmatrix}.
\end{align}
Since $|\ell\eta_j|\le B$, there exists $0<\delta<\frac{\pi}{2}$ such that $|\theta_\lambda|=|\arctan(-\ell\eta_j)|\le \delta$.
It follows that the diagonal entries are bounded by positive constants from above and below. Furthermore,
the off-diagonal entry is bounded from above in absolute value.
This leads to
\begin{align}\label{eq:shearindexest1}
|\langle e_\lambda, x_\lambda-x_\mu \rangle| &= |\langle R_{-\theta_\lambda}e_1, (S_{\ell,j})^{-1}A_{\alpha,\sigma^{-j}} k - x_\mu\rangle|
= |\langle e_1, TA_{\alpha,\sigma^{-j}} k -  R_{\theta_\lambda}x_\mu \rangle| \notag\\
&=|\sigma^{-j}k_1 \cos\theta_\lambda - [R_{\theta_\lambda}x_\mu]_1|  \asymp  |\sigma^{-j}k_1- \cos(\theta_\lambda)^{-1}[R_{\theta_\lambda}x_\mu]_1|.
\end{align}
It holds
\begin{align*}
|x_\lambda-x_\mu|=|S_{\ell,j}^{-1}A_{\alpha,\sigma^{-j}}k-x_\mu|=|\tilde{T}(R_{\theta_\lambda}^{-1}A_{\alpha,\sigma^{-j}}k-\tilde{T}^{-1}x_\mu)|
\end{align*}
where
\[
\tilde{T}=S_{\ell,j}^{-1}R_{\theta_\lambda}=\begin{pmatrix} \cos(\theta_\lambda)^{-1} & 0 \\ \sin\theta_\lambda & \cos(\theta_\lambda)  \end{pmatrix}
\]
is a `transfer matrix' similar to \eqref{eq:transfer}. We can conclude
\[
|x_\lambda-x_\mu| \asymp |A_{\alpha,\sigma^{-j}}k-R_{\theta_\lambda}\tilde{T}^{-1}x_\mu| \ge |s_\lambda^{-\alpha}k_2-[R_{\theta_\lambda}\tilde{T}^{-1}x_\mu]_2|.
\]

Now we distinguish between points $\mu\in\mathbb{P}$ with $|\theta_\mu|\le 2\arctan(B)$ and $|\theta_\mu|>2\arctan(B)$, where $B$
is the bound from \eqref{eq:shearbound}. We next require a simple result, which is as follows.

\begin{lemma}\label{lem:usefulobs}
For all $x,\,y\in\R$ absolutely bounded by some fixed bound $B\ge0$, i.e.\ $|x|,|y|\le B$, we have
\[
|\arctan x - \arctan y| \asymp |x - y|.
\]
\end{lemma}

\begin{proof}
For $x\neq y$ we have for some $\xi$ between $x$ and $y$ by the mean value theorem
\[
\frac{|\arctan x - \arctan y|}{|x - y|}=\arctan^\prime(\xi)=\frac{1}{1+\xi^2}.
\]
This yields
\[
\frac{1}{1+B^2} |x - y| \le |\arctan x - \arctan y| \le |x - y|.
\]
The case $x=y$ is trivial.
\end{proof}

As a consequence of this lemma, for $|\theta_\mu|\le 2\arctan(B)$, we obtain
\begin{align}\label{eq:shearindexest3}
\left|\arctan\left(-\ell\eta_j\right)-\theta_\mu\right|\asymp \left|-\ell\eta_j -\tan\theta_\mu\right| .
\end{align}
Since $\theta_\lambda=\arctan(-\ell\eta_j)$ and $|\ell\eta_j|\le B$, for $|\theta_\mu|>2\arctan(B)$, we have $|\theta_\mu|>2|\theta_\lambda|$. Thus, we obtain
\begin{align}\label{eq:shearindexest4}
\left|\arctan\left(-\ell\eta_j\right)-\theta_\mu\right| = \left|\theta_\lambda-\theta_\mu\right|\ge \left|\theta_\lambda\right| \asymp \left|\ell\eta_j\right|  .
\end{align}

In view of the estimates \eqref{eq:shearindexest1}-\eqref{eq:shearindexest3}, if $|\theta_\mu|\le2\arctan(B)$, we obtain
\beqn
\sum_{\substack{\lambda\in\Lambda^s \\ s_\lambda=\sigma^j, \varepsilon=0}}
(1+d_\alpha(\lambda,\mu))^{-N} &
\lesssim &
\sum_{k\in\Z^2}\sum_{\ell=- L_j}^{L_j}
\big(1+s_0^{2(1-\alpha)}|\arctan(- \ell\eta_j)-\theta_\mu |^2 \\
&& + s_0^{2\alpha}|(S_{\ell,j})^{-1} A_{\alpha,\sigma^{-j}}k - x_\mu|^2 +
s_0| \langle e_\lambda, (S_{\ell,j}^{-1})A_{\alpha,\sigma^{-j}}k- x_\mu\rangle |\big)^{-N} \\
&\lesssim&
\sum_{k\in\Z^2}\sum_{\ell=- L_j}^{L_j}
\big(1+s_0^{2(1-\alpha)}|- \ell\eta_j-\tan\theta_\mu |^2 + s_0^{2\alpha}|\sigma^{-j\alpha}k_2 - [R_{\theta_\lambda}\tilde{T}^{-1}x_\mu]_2|^2 \\
&&+ s_0| \sigma^{-j} k_1- [R_{\theta_\lambda}x_\mu]_1/\cos(\theta_\lambda) |\big)^{-N} \\
&\lesssim& \sum_{k\in\Z^2}\sum_{\ell\in\Z}
\big(1+| \ell\cdot (s_\lambda/s_0)^{-(1-\alpha)}-a_1 |^2 + |(s_\lambda/s_0)^{-\alpha}k_2-a_2(\ell) |^2 \\
&&+  | (s_\lambda/s_0)^{-1}k_1-a_3(\ell) | \big)^{-N}
\eeqn
with the quantities
\begin{align*}
a_1=-s_0^{(1-\alpha)}\tan\theta_\mu, &&
a_2(\ell)=s_0^\alpha[R_{\theta_\lambda}\tilde{T}^{-1}x_\mu]_2, &&
a_3(\ell)=s_0[R_{\theta_\lambda}x_\mu]_1/\cos(\theta_\lambda),
\end{align*}
depending on $j$, $\ell$ and $\mu$. This expression is similar to \eqref{eq:standsum}. Therefore from here we can proceed as in the
proof of part (i).

In case $|\theta_\mu|>2\arctan(B)$ we argue analogously, but we use \eqref{eq:shearindexest4} instead of \eqref{eq:shearindexest3}.
We obtain
\beqn
\sum_{\substack{\lambda\in\Lambda^s\\ s_\lambda=\sigma^j, \varepsilon=0}}
(1+d_\alpha(\lambda,\mu))^{-N}
&\lesssim&
\sum_{k\in\Z^2}\sum_{\ell=- L_j}^{L_j}
\Big(1+s_0^{2(1-\alpha)}|\ell\eta_j|^2 + s_0^{2\alpha}|(S_{\ell,j})^{-1} A_{\alpha,\sigma^{-j}}k - x_\mu|^2 \\
&& + s_0| \langle e_\lambda, (S_{\ell,j}^{-1})A_{\alpha,\sigma^{-j}}k- x_\mu\rangle |\Big)^{-N} \\
&\lesssim& \sum_{k\in\Z^2}\sum_{\ell\in\Z}
\big(1+| \ell\cdot (s_\lambda/s_0)^{-(1-\alpha)}|^2 + |(s_\lambda/s_0)^{-\alpha}k_2-a_2(\ell) |^2 \\
&&+  | (s_\lambda/s_0)^{-1}k_1-a_3(\ell) | \big)^{-N}
\eeqn
From here the proof again proceeds along the same lines as the proof of part (i).

\section*{Acknowledgements}

The first author was supported in part by Swiss National Fund (SNF) Grant 146356. The second author acknowledges support from the
Berlin Mathematical School and the DFG Collaborative Research Center TRR 109 ``Discretization in Geometry and Dynamics''.
The third author acknowledges support by the Einstein Foundation Berlin, by the Einstein Center for Mathematics Berlin
(ECMath), by Deutsche Forschungsgemeinschaft (DFG) Grant KU 1446/14, by the DFG Collaborative Research Center TRR 109
``Discretization in Geometry and Dynamics'', and by the DFG Research Center {\sc Matheon} ``Mathematics for Key Technologies''
in Berlin.

\bibliographystyle{abbrv}
\bibliography{molecules}

\end{document}